\documentclass[12pt, reqno]{amsart} 
\usepackage{amssymb,amscd,amsfonts,amsbsy}
\usepackage{latexsym}
\usepackage{exscale}
\usepackage{amsmath,amsthm,amsfonts}
\usepackage{mathrsfs}

\usepackage{xcolor} 
\usepackage[colorlinks=true,linkcolor=blue,citecolor=red,urlcolor=red]{hyperref} 
\usepackage{esint} 
\usepackage{amssymb} 

\calclayout

\allowdisplaybreaks[3] 

\newtheorem{theorem}{Theorem}[section]
\newtheorem{proposition}[theorem]{Proposition}
\newtheorem{lemma}[theorem]{Lemma}
\newtheorem{corollary}[theorem]{Corollary}

\theoremstyle{definition}
\newtheorem{definition}[theorem]{Definition}
\newtheorem{example}[theorem]{Example}
\theoremstyle{remark}
\newtheorem{remark}[theorem]{Remark}

\numberwithin{equation}{section} 

\def\BMO{\operatorname{BMO}}

\def\Re{\operatorname{Re}}
\def\inf{\operatornamewithlimits{inf\vphantom{p}}}

\def\loc{\operatorname{loc}}
\def\osc{\operatorname{osc}}
\def\Tr{\operatorname{Tr}}

\newcommand{\norm}[1]{\left\lVert#1\right\rVert}
\newcommand\restr[2]{\ensuremath{\left.#1\right|_{#2}}}

\begin{document}

\title[Generalized H\"{o}lder and Morrey-Campanato Dirichlet problems in $\mathbb{R}^n_{+}$]{The generalized H\"{o}lder and Morrey-Campanato Dirichlet problems for elliptic systems in the upper-half space}

\author{Juan Jos\'{e} Mar\'{i}n}
\address{Juan Jos\'{e} Mar\'{i}n
\\
Instituto de Ciencias Matem\'aticas CSIC-UAM-UC3M-UCM
\\
Consejo Superior de Investigaciones Cient{\'\i}ficas
\\
C/ Nicol\'as Cabrera, 13-15
\\
E-28049 Madrid, Spain} \email{juanjose.marin@icmat.es}

\author{Jos\'e Mar{\'\i}a Martell}
\address{Jos\'e Mar{\'\i}a Martell
\\
Instituto de Ciencias Matem\'aticas CSIC-UAM-UC3M-UCM
\\
Consejo Superior de Investigaciones Cient{\'\i}ficas
\\
C/ Nicol\'as Cabrera, 13-15
\\
E-28049 Madrid, Spain} \email{chema.martell@icmat.es}

\author{Marius Mitrea}
\address{Marius Mitrea
\\
Department of Mathematics
\\
Baylor University 
\\
Waco, TX 76706, USA} 
\email{Marius$\underline{\,\,\,}$\,Mitrea@baylor.edu}

\thanks{The first and second authors acknowledge financial support from the Spanish Ministry of Economy and Competitiveness, through the ``Severo Ochoa Programme for Centres of Excellence in R\&D'' (SEV-2015-0554). They also acknowledge that
the research leading to these results has received funding from the European Research
Council under the European Union's Seventh Framework Programme (FP7/2007-2013)/ ERC
agreement no. 615112 HAPDEGMT. The last author has been
supported in part by the Simons Foundation grant $\#\,$281566.}

\date{April 20, 2018. \textit{Revised:} \today.}

\subjclass[2010]{Primary: 35B65, 35C15, 35J47, 35J57, 35J67, 42B37. 
Secondary: 35E99, 42B35.}

\keywords{Generalized H\"older space, generalized Morrey-Campanato  space, Dirichlet problem in the upper-half space, second order elliptic system, Poisson kernel, Lam\'e system,  nontangential pointwise trace, Fatou type theorem}

\begin{abstract}
We prove well-posedness results for the Dirichlet problem in $\mathbb{R}^{n}_{+}$ for homogeneous, 
second order, constant complex coefficient elliptic systems with boundary data in generalized 
H\"older spaces $\mathscr{C}^{\omega}(\mathbb{R}^{n-1},\mathbb{C}^M)$ and in generalized Morrey-Campanato 
spaces $\mathscr{E}^{\omega,p}(\mathbb{R}^{n-1},\mathbb{C}^M)$ under certain assumptions on the growth 
function $\omega$. We also identify a class of growth functions $\omega$ for which 
$\mathscr{C}^{\omega}(\mathbb{R}^{n-1},\mathbb{C}^M)=\mathscr{E}^{\omega,p}(\mathbb{R}^{n-1},\mathbb{C}^M)$ 
and for which the aforementioned well-posedness results are equivalent, in the sense that they have 
the same unique solution, satisfying natural regularity properties and estimates.

\end{abstract}

\maketitle
\tableofcontents

\section{Introduction}

This paper is devoted to studying the Dirichlet problem for elliptic systems in 
the upper-half space with data in generalized H\"{o}lder spaces and generalized 
Morrey-Campa\-nato spaces. As a byproduct of the PDE-based techniques developed here, we are able
to establish the equivalence of these function spaces. To be more specific requires introducing some notation. 

Having fixed $n,M\in{\mathbb{N}}$ with $n\geq 2$ and $M\geq 1$, consider a homogeneous, 
constant (complex) coefficient, $M\times M$ second-order system in ${\mathbb{R}}^n$
\begin{equation}\label{ars}
L:=\left(a_{jk}^{\alpha\beta}\partial_j\partial_k\right)_{1\leq\alpha,\beta\leq M}.
\end{equation} 
Here and elsewhere, the summation convention over repeated indices is employed. 
We make the standing assumption that $L$ is strongly elliptic, in the sense that 
there exists $\kappa_0>0$ such that the following Legendre-Hadamard condition is satisfied: 
\begin{equation}\label{LegHad} 
\begin{array}{c}
\Re\left[a_{jk}^{\alpha\beta}\xi_j\xi_k\overline{\zeta_\alpha}\zeta_\beta\right]
\geq\kappa_0\left|\xi\right|^2\left|\zeta\right|^2\,\,\text{ for all}
\\[10pt]
\xi=\left(\xi_j\right)_{1\leq j\leq n}\in\mathbb{R}^n
\,\,\text{ and }\,\,\zeta=\left(\zeta_\alpha\right)_{1\leq\alpha\leq M}\in\mathbb{C}^M.
\end{array}
\end{equation}
Examples include scalar operators, such as the Laplacian $\Delta=\sum\limits_{j=1}^n\partial_j^2$ 
or, more generally, operators of the form ${\rm div}A\nabla$ with $A=(a_{rs})_{1\leq r,s\leq n}$ an 
$n\times n$ matrix with complex entries satisfying the ellipticity condition
\begin{equation}\label{YUjhv-753}
\inf_{\xi\in S^{n-1}}{\rm Re}\,\big[a_{rs}\xi_r\xi_s\bigr]>0,
\end{equation}
(where $S^{n-1}$ denotes the unit sphere in ${\mathbb{R}}^n$), as well as the complex
version of the Lam\'e system of elasticity in $\mathbb{R}^n$, 
\begin{equation}\label{TYd-YG-76g}
L:=\mu\Delta+(\lambda+\mu)\nabla{\rm div}.
\end{equation}
Above, the constants $\lambda,\mu\in{\mathbb{C}}$ (typically called Lam\'e moduli),
are assumed to satisfy
\begin{equation}\label{Yfhv-8yg}
{\rm Re}\,\mu>0\,\,\mbox{ and }\,\,{\rm Re}\,(2\mu+\lambda)>0,
\end{equation}
a condition equivalent to the demand that the Lam\'e system \eqref{TYd-YG-76g} 
satisfies the Legendre-Hadamard ellipticity condition \eqref{LegHad}.
While the Lam\'e system is symmetric, we stress that the results in this paper require no symmetry 
for the systems involved. 

With each system $L$ as in \eqref{ars}-\eqref{LegHad} one may associate a Poisson kernel, $P^L$, 
which is a ${\mathbb{C}}^{M\times M}$-valued function defined in ${\mathbb{R}}^{n-1}$ described 
in detail in Theorem~\ref{PoissonConvolution}. This Poisson kernel has played a pivotal role in 
the treatment of the Dirichlet problem with data in $L^p$, $\mathrm{BMO}$, $\mathrm{VMO}$ and 
H\"older spaces (see \cite{K-MMMM, BMO-MarMitMitMit16}). For now, we make the observation that the 
Poisson kernel gives rise to a nice approximation to the identity in ${\mathbb{R}}^{n-1}$ by setting 
$P^L_t(x')=t^{1-n}P^L(x'/t)$ for every $x'\in{\mathbb{R}}^{n-1}$ and $t>0$.

For every point $x\in\mathbb{R}^n$ write $x=(x',t)$, where $x'\in\mathbb{R}^{n-1}$ corresponds to 
the first $n-1$ coordinates of $x$, and $t\in\mathbb{R}$ is the last coordinate of $x$. As is customary, 
we shall let $\mathbb{R}^n_{+}:=\{x=(x',t)\in\mathbb{R}^n:\,x'\in\mathbb{R}^{n-1},\,t>0\}$ denote the 
upper-half space in $\mathbb{R}^n$, and typically identify its boundary with $(n-1)$-dimensional Euclidean 
space, via $\partial\mathbb{R}^n_{+}\ni(x',0)\equiv x'\in\mathbb{R}^{n-1}$. The cone with vertex at 
$x'\in\mathbb{R}^{n-1}$ and aperture $\kappa>0$ is defined as 
\begin{equation}\label{64rrf}
\Gamma_{\kappa}(x'):=\{y=(y',t)\in\mathbb{R}^n_{+}:\,|x'-y'|<\kappa t\}.
\end{equation} 
When $\kappa=1$ we agree to drop the dependence on aperture and simply write $\Gamma(x')$. 
Whenever meaningful, the nontangential pointwise trace of a vector-valued function $u$ defined 
in $\mathbb{R}^n_{+}$ is given by 
\begin{equation}\label{6tfF}
\left(\restr{u}{\partial\mathbb{R}^n_{+}}^{{}^{\rm nt.lim}}\right)(x'):= 
\lim_{\substack{\mathbb{R}^n_+\ni y\to (x',0)\\ y\in\Gamma_\kappa (x')}}u(y),
\quad x'\in\mathbb{R}^{n-1}. 
\end{equation}
The unrestricted pointwise trace of a vector-valued function $u$ defined in 
$\mathbb{R}^n_+$ at each $x'\in\partial\mathbb{R}^n_+\equiv\mathbb{R}^{n-1}$ is taken to be 
\begin{equation}\label{tfDSq56f}
\left(\restr{u}{\partial\mathbb{R}^n_{+}}^{{}^{\rm lim}}\right)(x')
:=\lim_{\mathbb{R}^n_+\ni y\to(x',0)}u(y),\quad x'\in\mathbb{R}^{n-1},
\end{equation}
whenever such a limit exists exists. 

\begin{definition}\label{6ttFSD}
Call a given mapping $\omega:(0,\infty)\to(0,\infty)$ a \texttt{growth function} if $\omega$ is 
non-decreasing and $\omega(t)\to 0$ as $t\to 0^{+}$.
\end{definition}

\begin{definition}\label{7tFFD}
Let $E\subset\mathbb{R}^{n}$ be an arbitrary set (implicitly assumed to have cardinality at least $2$) 
and let $\omega$ be a growth function. The \texttt{homogeneous $\omega$-H\"older space} on $E$ is defined as
\begin{equation}\label{7rsSS}
\dot{\mathscr{C}}^\omega(E,\mathbb{C}^M):=\big\{u:E\to\mathbb{C}^M:\,
[u]_{\dot{\mathscr{C}}^\omega(E,\mathbb{C}^M)}<\infty\big\},
\end{equation}
where $[\cdot]_{\dot{\mathscr{C}}^\omega(E,\mathbb{C}^M)}$ stands for the seminorm
\begin{equation}\label{7tEEE}
[u]_{\dot{\mathscr{C}}^\omega(E,\mathbb{C}^M)}:=\sup_{\substack{x,y\in E\\ x\neq y}}\frac{|u(x)-u(y)|}{\omega(|x-y|)}.
\end{equation}
\end{definition}

Let us note that the fact that $\omega(t)\to 0$ as $t\to 0^+$ implies that if 
$u\in\dot{\mathscr{C}}^\omega(E,\mathbb{C}^M)$ then $u$ is uniformly continuous. 
The choice $\omega(t):=t^\alpha$ for each $t>0$, with $\alpha\in(0,1)$, yields the 
classical scale of H\"older spaces.

Here and elsewhere in the paper, we agree to denote the $(n-1)$-dimensional Lebesgue measure of 
given Lebesgue measurable set $E\subseteq{\mathbb{R}}^{n-1}$ by $|E|$. Also, by a cube $Q$ in $\mathbb{R}^{n-1}$ 
we shall understand a cube with sides parallel to the coordinate axes. Its side-length will be denoted by $\ell(Q)$, 
and for each $\lambda>0$ we shall denote by $\lambda\,Q$ the cube concentric with $Q$ whose side-length is 
$\lambda\,\ell(Q)$. For every function $h\in L^1_{\rm loc}(\mathbb{R}^{n-1},{\mathbb{C}}^M)$ we write 
\begin{equation}\label{nota-aver}
h_Q:=\fint_{Q} h(x')\,dx':=\frac{1}{|Q|}\int_{Q}h(x')\,dx'\in{\mathbb{C}}^M,
\end{equation}
with the integration performed componentwise. 

\begin{definition}\label{65RFF}
Given a growth function $\omega$ along with some integrability exponent $p\in[1,\infty)$, 
the associated \texttt{generalized Morrey-Campanato space} in $\mathbb{R}^{n-1}$ is defined as
\begin{equation}\label{utfFFD}
\mathscr{E}^{\omega,p}(\mathbb{R}^{n-1},\mathbb{C}^M)
:=\big\{f\in L^1_{\loc}(\mathbb{R}^{n-1},\mathbb{C}^M):\,
\norm{f}_{\mathscr{E}^{\omega,p}(\mathbb{R}^{n-1},\mathbb{C}^M)}<\infty\big\},
\end{equation}
where $\norm{f}_{\mathscr{E}^{\omega,p}(\mathbb{R}^{n-1},\mathbb{C}^M)}$ stands for the seminorm
\begin{equation}\label{987tF}
\norm{f}_{\mathscr{E}^{\omega,p}(\mathbb{R}^{n-1},\mathbb{C}^M)}
:=\sup_{Q\subset\mathbb{R}^{n-1}}\frac{1}{\omega(\ell(Q))}\bigg(\fint_Q|f(x')-f_Q|^p\,dx'\bigg)^{1/p}.
\end{equation}
\end{definition}
The choice $\omega(t):=t^\alpha$ with $\alpha\in(0,1)$ corresponds to the classical Morrey-Campanato spaces, 
while the special case $\omega(t):=1$ yields the usual space of functions of bounded mean oscillations ($\BMO$). 
We also define, for every $u\in\mathscr{C}^1(\mathbb{R}^{n}_{+},\mathbb{C}^M)$ and $q\in(0,\infty)$, 
\begin{equation}\label{7yGGG}
\norm{u}_{**}^{(\omega,q)}:=\sup_{Q\subset\mathbb{R}^{n-1}}
\frac{1}{\omega(\ell(Q))}\bigg(\fint_Q\bigg(\int_0^{\ell(Q)}|(\nabla u)(x',t)|^2\,t\,dt\bigg)^{q/2}\,dx'\bigg)^{1/q}.
\end{equation}
As far as this seminorm is concerned, there are two reasonable candidates for the end-point $q=\infty$  
(see Proposition~\ref{prop:proper-sols} and Lemma~\ref{lem:appendix}). First, we may consider 
\begin{equation}\label{ygFF}
\norm{u}_{**}^{(\omega,\exp)}:=\sup_{Q\subset\mathbb{R}^{n-1}}
\frac{1}{\omega(\ell(Q))}\bigg\|  \Big(  \int_0^{\ell(Q)}|(\nabla u)(\cdot,t)|^2\,t\,dt
\Big)^{1/2} \bigg\|_{\exp L,Q}
\end{equation}
where $\|\cdot\|_{\exp L,Q}$ is the version of the norm in the Orlicz space $\exp L$ 
localized and normalized relative to $Q$, i.e., 
\begin{equation}\label{eq:Lux-norm}
\|f\|_{\exp L,Q}:=\inf\left\{t>0:\fint_{Q}\Big(e^{\tfrac{|f(x')|}{t}}-1\Big)\,dx'\le 1\right\}.
\end{equation}
Second, corresponding to the limiting case $q=\infty$ we may consider
\begin{equation}\label{yRRFF}
\norm{u}_{**}^{(\omega,\infty)}:=\sup_{(x',t)\in\mathbb{R}^{n}_{+}}\frac{t}{\omega(t)}|(\nabla u)(x',t)|.
\end{equation}

We are ready to describe our main result concerning the Dirichlet problems with 
data in generalized H\"older and generalized Morrey-Campanato spaces for homogeneous 
second-order strongly elliptic systems of differential operators with constant complex coefficients 
(cf. \eqref{ars} and \eqref{LegHad}). In Section~\ref{section:well-general} 
(cf. Theorems~\ref{mainthmBMO2b}-\ref{mainthmBMO2}), we weaken the condition \eqref{omega-cond:main} 
and still prove well-posedness for the two Dirichlet problems. The main difference 
is that in that case they are no longer equivalent as \eqref{maincorBMOeq} might fail 
(see Example~\ref{example}). 

\begin{theorem}\label{mainthmBMO}
Consider a strongly elliptic constant complex coefficient second-order $M\times M$ system $L$, as in 
\eqref{ars}-\eqref{LegHad}. Also, fix $p\in[1,\infty)$ along with $q\in(0,\infty]$, and let $\omega$ 
be a growth function satisfying, for some finite constant $C_0\geq 1$, 
\begin{equation}\label{omega-cond:main}
\int_0^{t}\omega(s)\frac{ds}{s}+t\,\int_t^{\infty}\frac{\omega(s)}{s}\,\frac{ds}{s}\leq C_0\,\omega(t)
\,\,\text{ for each }\,\,t\in(0,\infty).
\end{equation}

Then the following statements are true. 

\begin{list}{\textup{(\theenumi)}}{\usecounter{enumi}\leftmargin=1cm \labelwidth=1cm \itemsep=0.2cm \topsep=.2cm \renewcommand{\theenumi}{\alph{enumi}}}

\item\label{bvp-Hol-Dir:main} 
The generalized H\"older Dirichlet problem for the system $L$ in $\mathbb{R}^{n}_{+}$, i.e., 
\begin{equation}\label{BVPb}
\left\lbrace
\begin{array}{l}
u\in\mathscr{C}^{\infty}(\mathbb{R}^{n}_{+},\mathbb{C}^M),
\\[4pt]
Lu=0\,\,\text{ in }\,\,\mathbb{R}^n_{+},
\\[4pt] 
\left[u\right]_{\dot{\mathscr{C}}^{\omega}(\mathbb{R}^n_{+},\mathbb{C}^M)}<\infty,
\\[4pt]
\restr{u}{\partial\mathbb{R}^n_{+}}^{{}^{\rm lim}}=f\in\dot{\mathscr{C}}^{\omega}(\mathbb{R}^{n-1},\mathbb{C}^M)
\,\,\text{ on }\,\,\mathbb{R}^{n-1},
\end{array}
\right.
\end{equation}
is well-posed. More specifically, there is a unique solution which is given by 
\begin{equation}\label{eqn-Dir-BMO:u}
u(x',t)=(P_t^L*f)(x'),\qquad\forall\,(x',t)\in{\mathbb{R}}^n_{+},
\end{equation}
where $P^L$ denotes the Poisson kernel for $L$ in $\mathbb{R}^{n}_+$ from Theorem~\ref{PoissonConvolution}. 
In addition, $u$ belongs to the space $\dot{\mathscr{C}}^{\omega}(\overline{\mathbb{R}^n_{+}},\mathbb{C}^M)$, 
satisfies $\restr{u}{\partial\mathbb{R}^n_{+}}=f$, and there exists a finite constant  $C=C(n,L,\omega)\geq 1$ 
such that
\begin{equation}\label{mainthmBMOeq2}
C^{-1}[f]_{\dot{\mathscr{C}}^{\omega}(\mathbb{R}^{n-1},\mathbb{C}^M)} 
\leq[u]_{\dot{\mathscr{C}}^{\omega}(\mathbb{R}^n_{+},\mathbb{C}^M)} 
\leq C [f]_{\dot{\mathscr{C}}^{\omega}(\mathbb{R}^{n-1},\mathbb{C}^M)}.
\end{equation}  

\item\label{bvp-MC-Dir:main} 
The generalized Morrey-Campanato Dirichlet problem for $L$ in $\mathbb{R}^{n}_{+}$, formulated as
\begin{equation}\label{BVP}
\left\lbrace
\begin{array}{l}
u\in\mathscr{C}^{\infty}(\mathbb{R}^{n}_{+},\mathbb{C}^M),
\\[4pt]
Lu=0\,\,\text{ in }\,\,\mathbb{R}^n_{+},
\\[4pt] 
\norm{u}_{**}^{(\omega,q)}<\infty,
\\[4pt]
\restr{u}{\partial\mathbb{R}^n_{+}}^{{}^{\rm nt.lim}}
=f\in\mathscr{E}^{\omega,p}(\mathbb{R}^{n-1},\mathbb{C}^M)\,\,\text{ a.e. on }\,\,\mathbb{R}^{n-1},
\end{array}
\right.
\end{equation}
is well-posed. More precisely, there is a unique solution \eqref{BVP} which is given by \eqref{eqn-Dir-BMO:u}. 
In addition, $u$ belongs to $\dot{\mathscr{C}}^{\omega}(\overline{\mathbb{R}^n_{+}},\mathbb{C}^M)$, satisfies
$\restr{u}{\partial\mathbb{R}^n_{+}}=f$ a.e. on $\mathbb{R}^{n-1}$, and there exists a 
finite constant $C=C(n,L,\omega,p,q)\geq 1$ such that
\begin{equation}\label{mainthmBMOeq1}
C^{-1}\norm{f}_{\mathscr{E}^{\omega,p}(\mathbb{R}^{n-1},\mathbb{C}^M)} 
\leq\norm{u}_{**}^{(\omega,q)} 
\leq C\norm{f}_{\mathscr{E}^{\omega,p}(\mathbb{R}^{n-1},\mathbb{C}^M)}.
\end{equation} 
Furthermore, all these properties remain true if $\norm{\cdot}_{**}^{(\omega,q)}$ is replaced everywhere 
by $\norm{\cdot}_{**}^{(\omega,\exp)}$.

\item\label{equiv:main} 
The following equality between vector spaces holds 
\begin{equation}\label{maincorBMOeq}
\dot{\mathscr{C}}^{\omega}(\mathbb{R}^{n-1},\mathbb{C}^M)
=\mathscr{E}^{\omega,p}(\mathbb{R}^{n-1},\mathbb{C}^M)  
\end{equation}
with equivalent norms, where the right-to-left inclusion is understood in the sense 
that for each $f\in\mathscr{E}^{\omega,p}(\mathbb{R}^{n-1},\mathbb{C}^M)$ there exists a unique 
$\widetilde{f}\in\dot{\mathscr{C}}^{\omega}(\mathbb{R}^{n-1},\mathbb{C}^M)$ with the property that 
$f=\widetilde{f}$ a.e. in $\mathbb{R}^{n-1}$. 

As a result, the Dirichlet problems \eqref{BVPb} and \eqref{BVP} are equivalent. Specifically, 
for any pair of boundary data which may be identified in the sense of \eqref{maincorBMOeq}
these problems have the same unique solution (given by \eqref{eqn-Dir-BMO:u}). 
\end{list}
\end{theorem}

A few comments regarding the previous result. In Lemma~\ref{wlemma} we shall 
prove that, for growth functions as in \eqref{omega-cond:main}, each 
$u\in\dot{\mathscr{C}}^{\omega}(\mathbb{R}^n_{+},\mathbb{C}^M)$ extends uniquely to 
a function $u\in\dot{\mathscr{C}}^{\omega}(\overline{\mathbb{R}^n_{+}},\mathbb{C}^M)$. 
Hence, the ordinary restriction $\restr{u}{\partial\mathbb{R}^n_{+}}$ is well-defined in the
context of item {\rm (a)} of Theorem~\ref{mainthmBMO}. 
In item {\rm (b)} the situation is slightly different. One can first show that $u$ extends to a 
continuous function up to, and including, the boundary. Hence, the non-tangential pointwise 
trace agrees with the restriction to the boundary everywhere. However, since functions in 
$\mathscr{E}^{\omega,p}(\mathbb{R}^{n-1},\mathbb{C}^M)$ are canonically identified whenever they 
agree outside of a set of zero Lebesgue measure, the boundary condition in \eqref{BVP} is most naturally 
formulated by asking that the non-tangential boundary trace agrees with the boundary datum almost everywhere.
The same type of issue arises when interpreting \eqref{maincorBMOeq}. Specifically, while the left-to-right inclusion 
has a clear meaning, the converse inclusion should be interpreted as saying that each equivalence class
in $\mathscr{E}^{\omega,p}(\mathbb{R}^{n-1},\mathbb{C}^M)$ (induced by the aforementioned identification) 
has a unique representative from $\dot{\mathscr{C}}^{\omega}(\mathbb{R}^{n-1},\mathbb{C}^M)$.   We would like to observe that \eqref{maincorBMOeq} extends the well-known result of N.G.~Meyers \cite{Mey} who considered the case $\omega(t)=t^\alpha$, $t>0$. Here we extend the class of growth functions for which \eqref{maincorBMOeq} holds and our  alternative approach is based on PDE. 
 
It is illustrative to provide examples of growth functions to which Theorem~\ref{mainthmBMO} applies. 
In this vein, we first observe that \eqref{omega-cond:main} is closely related to the dilation indices 
of Orlicz spaces studied in \cite{BenSha88, FioKrb97} in relation to interpolation in Orlicz spaces. 
Concretely, given a growth function $\omega$ set 
\begin{equation}\label{6gDDD.1} 
h_{\omega}(t):=\sup_{s>0}\frac{\omega(st)}{\omega(s)},\qquad\forall\,t>0, 
\end{equation} 
and define the lower and upper dilation indices, respectively, as
\begin{equation}\label{6gDDD.2}
i_{\omega}:=\sup_{0<t<1}\frac{\log h_{\omega}(t)}{\log t}\,\,\text{ and }\,\,
I_{\omega}:=\inf_{t>1}\frac{\log h_{\omega}(t)}{\log t}. 
\end{equation} 
One can see that if $0<i_\omega\leq I_\omega<1$ then \eqref{omega-cond:main} holds. 
Indeed, it is not hard to check that there exists a constant $C\in(0,\infty)$ with the property that
$h_\omega(t)\leq C\,t^{i_\omega/2}$ for every $t\in(0,1]$, and $h_\omega(t)\leq C\,t^{(I_\omega+1)/2}$ 
for every $t\in[1,\infty)$. These, in turn, readily yield \eqref{omega-cond:main}.

Now, given $\alpha\in(0,1)$, if $\omega(t):=t^\alpha$ for each $t>0$ then $i_\omega=I_\omega=\alpha$ and, hence, 
\eqref{omega-cond:main} holds. Note that in that case $\dot{\mathscr{C}}^{\omega}(\mathbb{R}^n_{+},\mathbb{C}^M)
=\dot{\mathscr{C}}^{\alpha }(\mathbb{R}^n_{+},\mathbb{C}^M)$ is the standard homogeneous H\"older space of order 
$\alpha$ and the particular version of Theorem~\ref{mainthmBMO} corresponding to this scenario 
has been established in \cite{BMO-MarMitMitMit16}. This being said, there many examples of interest that are treated 
here for the first time, such as $\omega(t) = t^\alpha\,(A+\log_{+}t)^\theta$ for $A:=\max \{ 1, -\theta/\alpha \}$ and each $t>0$, or 
$\omega(t)= t^\alpha\,(A+\log_{+}(1/t))^\theta$ for $A:=\max \{ 1, \theta/\alpha \}$ and each $t>0$, with $0<\alpha<1$, $\theta\in\mathbb{R}$,
and $\log_{+}(t) := \max \{ 0 , \log t \}$.

In these situations $i_\omega=I_\omega=\alpha$ which guarantees that \eqref{omega-cond:main} holds. 
Furthermore, if $\omega(t):=\max\{t^\alpha,t^\beta\}$, or $\omega(t):=\min\{t^\alpha,t^\beta\}$, 
for each $t>0$, with $0<\alpha,\beta<1$, then in both cases we have $i_\omega=\min\{\alpha,\beta\}$ 
and $I_\omega=\max\{\alpha,\beta\}$, hence condition \eqref{omega-cond:main} is verified once again. 

The following result, providing a characterization of the generalized H\"older and generalized Morrey-Campanato 
spaces in terms of the boundary traces of solutions, is a byproduct of the proof of the above theorem.

\begin{corollary}\label{maincorBMO}
Let $L$ be a strongly elliptic, constant {\rm (}complex{\rm )} coefficient, second-order 
$M\times M$ system in ${\mathbb{R}}^n$. Fix $p\in[1,\infty)$ along with $q\in(0,\infty)$, 
and let $\omega$ be a growth function for which \eqref{omega-cond:main} holds. Then for every function 
$u\in\mathscr{C}^{\infty}(\mathbb{R}^n_{+},\mathbb{C}^M)$ satisfying $Lu=0$ in $\mathbb{R}^{n}_{+}$ 
one has 
\begin{equation}\label{eq:qfrafr}
\norm{u}_{**}^{(\omega,q)}\approx\norm{u}_{**}^{(\omega,\exp)}\approx\norm{u}_{**}^{(\omega,\infty)}
\approx[u]_{\dot{\mathscr{C}}^{w}(\mathbb{R}^{n}_{+},\mathbb{C}^M)} 
\end{equation}
where the implicit proportionality constants depend only on $L$, $n$, $q$, and the constant 
$C_0$ in \eqref{omega-cond:main}. Moreover, 
\begin{align}\label{HLMO}
\dot{\mathscr{C}}^{\omega}(\mathbb{R}^{n-1},\mathbb{C}^M) &=\big\{\restr{u}{\partial\mathbb{R}^n_{+}}:\,
u\in\mathscr{C}^{\infty}(\mathbb{R}^n_{+},\mathbb{C}^M),\,\,Lu=0\text{ in }\mathbb{R}^{n}_{+},\,\,
\left[u\right]_{\dot{\mathscr{C}}^{\omega}(\mathbb{R}^n_{+},\mathbb{C}^M)}<\infty\big\}
\nonumber\\[7pt]
&=\big\{\restr{u}{\partial\mathbb{R}^n_{+}}:\,u\in\mathscr{C}^{\infty}(\mathbb{R}^n_{+},\mathbb{C}^M), 
\,\,Lu=0\text{ in }\mathbb{R}^{n}_{+},\,\,\norm{u}_{**}^{(\omega,q)}<\infty\big\}
\nonumber\\[7pt]
&=\big\{\restr{u}{\partial\mathbb{R}^n_{+}}:\,u\in\mathscr{C}^{\infty}(\mathbb{R}^n_{+},\mathbb{C}^M), 
\,\,Lu=0\text{ in }\mathbb{R}^{n}_{+},\,\,\norm{u}_{**}^{(\omega,\exp)}<\infty\big\}
\nonumber\\[7pt] 
&=\big\{\restr{u}{\partial\mathbb{R}^n_{+}}:\,u\in\mathscr{C}^{\infty}(\mathbb{R}^n_{+},\mathbb{C}^M), 
\,\,Lu=0\text{ in }\mathbb{R}^{n}_{+},\,\,\norm{u}_{**}^{(\omega,\infty)}<\infty\big\}.
\end{align} 
\end{corollary}

The plan of the paper is as follows. In Section~\ref{section:growth} we present some properties 
of the growth functions and study some of the features of the generalized H\"{o}lder and 
Morrey-Campanato spaces which are relevant to this work. Section~\ref{section:props-elliptic} 
is reserved for collecting some known results for elliptic systems, and for giving the proof 
of Proposition~\ref{prop:proper-sols}, where some a priori estimates for the null-solutions 
of such systems are established. In turn, these estimates allow us to compare the seminorm 
$\norm{\cdot}_{**}^{(\omega,q)}$ (corresponding to various values of $q$) with 
$[\cdot]_{\dot{\mathscr{C}}^{w}(\mathbb{R}^{n}_{+},\mathbb{C}^M)}$. 
In Section~\ref{section:Existence} we prove the existence of solutions for the Dirichlet problems 
with boundary data in $\dot{\mathscr{C}}^{\omega}(\mathbb{R}^{n-1},\mathbb{C}^M)$ and 
$\mathscr{E}^{\omega,p}(\mathbb{R}^{n-1},\mathbb{C}^M)$. Section~\ref{section:Fatou} 
contains a Fatou-type result for null-solutions of a strongly elliptic system $L$  
belonging to the space $\dot{\mathscr{C}}^{\omega}(\mathbb{R}^{n-1},\mathbb{C}^M)$, 
which will be a key ingredient when establishing uniqueness for the boundary value problems
formulated in Theorem~\ref{mainthmBMO}. Combining the main results of the previous two sections yields 
two well-posedness results under different assumptions on the growth function: one for boundary data in 
$\dot{\mathscr{C}}^{\omega}(\mathbb{R}^{n-1},\mathbb{C}^M)$ and solutions in 
$\dot{\mathscr{C}}^{\omega}(\mathbb{R}^{n}_{+},\mathbb{C}^M)$, and another one for 
boundary data in $\mathscr{E}^{\omega,p}(\mathbb{R}^{n-1},\mathbb{C}^M)$ and solutions satisfying 
$\norm{u}_{**}^{(\omega,q)}<\infty$ for some $0<q\le\infty$, or even in the case where $q$ 
is replaced by $\exp$. In concert, these two results cover all claims of Theorem~\ref{mainthmBMO}. 
Finally, in Appendix~\ref{section:JN} we present a John-Nirenberg type inequality of real-variable 
nature, generalizing some results in \cite{Hofmann-Mayboroda, HofMarMay} by allowing more flexibility 
due to the involvement of growth functions. This is interesting and useful in its own right. 
In addition, we are able to show exponential decay for the measure of the associated level sets 
which, in turn, permits deriving estimates not only in arbitrary $L^q$ spaces but also in the space 
$\exp L$. Our approach for deriving such results is different from \cite{Hofmann-Mayboroda, HofMarMay},
and uses some ideas which go back to a proof of the classical John-Nirenberg exponential integrability 
for {\rm BMO} functions due to Calder\'{o}n. As a matter of fact, our abstract method yields easily 
Calder\'on's classical result.

\section{Growth Functions, Generalized H\"{o}lder and  Morrey-Campanato Spaces}\label{section:growth}

We begin by studying some basic properties of growth functions. As explained in the introduction, we ultimately 
wish to work with growth functions satisfying conditions weaker than \eqref{omega-cond:main}. 
Indeed, the two mains conditions that we will consider are
\begin{equation}\label{omega-cond:a}
\int_0^{1}\omega(s)\frac{ds}{s}<\infty,
\end{equation}
and 
\begin{equation}\label{omega-cond:b}
t\,\int_t^{\infty}\frac{\omega(s)}{s}\,\frac{ds}{s}\leq C_\omega\,\omega(t),\qquad\forall\,t\in(0,\infty),
\end{equation}
for some finite constant $C_\omega\geq 1$. In what follows, $C_\omega$ will always denote the constant 
in \eqref{omega-cond:b}. Clearly, if $\omega$ satisfies it satisfies \eqref{omega-cond:main} then both 
\eqref{omega-cond:a} and \eqref{omega-cond:b} hold but the reverse implication is not true in general 
(see Example~\ref{example} in this regard). 

Later on, we will need the auxiliary function $W$ defined as 
\begin{equation}\label{omega-cond:WDef}
W(t):=\int_0^{t}\omega(s)\frac{ds}{s}\,\,\text{ for each }\,\,t\in(0,\infty).
\end{equation}
Note that \eqref{omega-cond:a} gives that $W(t)<\infty$ for every $t>0$. Then \eqref{omega-cond:main} 
holds if and only if \eqref{omega-cond:b} holds and there exists $C\in(0,\infty)$ such that 
$W(t)\leq C\,\omega(t)$ for each $t\in(0,\infty)$. 

The following lemma gathers some useful properties on growth functions satisfying 
condition \eqref{omega-cond:b}.

\begin{lemma}\label{wlemma}
Given a growth function $\omega$ satisfying \eqref{omega-cond:b}, the following statements are true. 

\begin{list}{\textup{(\theenumi)}}{\usecounter{enumi}\leftmargin=1cm \labelwidth=1cm \itemsep=0.2cm \topsep=.2cm \renewcommand{\theenumi}{\alph{enumi}}}

\item\label{omega-t-incre} 
Whenever $0<t_1\leq t_2<\infty$, one has
\begin{equation}\label{wnonincreasing}
\frac{\omega(t_2)}{t_2}\leq C_\omega\frac{\omega(t_1)}{t_1}.
\end{equation}

\item\label{item:wdoubling} 
For every $t\in(0,\infty)$ one has 
\begin{equation}\label{wdoubling}
\omega(2t)\leq 2C_\omega\,\omega(t).
\end{equation} 

\item\label{wlimit}
One has $\lim_{t\to\infty}\omega(t)/t=0$. 

\item\label{holderclosure} 
For each set $E\subset\mathbb{R}^n$ one has
$\dot{\mathscr{C}}^\omega(E,\mathbb{C}^M)=\dot{\mathscr{C}}^\omega(\overline{E},\mathbb{C}^M)$, 
with equivalent norms. More specifically, the restriction map 
\begin{equation}\label{tgVVVa}
\dot{\mathscr{C}}^\omega(\overline{E},\mathbb{C}^M)\ni u\longmapsto u\big|_{E}
\in\dot{\mathscr{C}}^\omega(E,\mathbb{C}^M)
\end{equation}
is a linear isomorphism which is continuous in the precise sense that, under the canonically identification
of functions $u\in\dot{\mathscr{C}}^\omega(\overline{E},\mathbb{C}^M)$ with 
$u\big|_{E}\in\dot{\mathscr{C}}^\omega(E,\mathbb{C}^M)$, one has
\begin{equation}\label{holderclosureeq}
[u]_{\dot{\mathscr{C}}^\omega(E,\mathbb{C}^M)}
\leq[u]_{\dot{\mathscr{C}}^\omega(\overline{E},\mathbb{C}^M)} 
\leq 2C_\omega[u]_{\dot{\mathscr{C}}^\omega(E,\mathbb{C}^M)}
\end{equation}
for each $u\in\dot{\mathscr{C}}^\omega(E,\mathbb{C}^M)$.
\end{list}
\end{lemma}

\begin{proof} 
We start observing that for every $t>0$
\begin{equation}\label{eq:fafr}
\frac{\omega(t)}{t}\leq\int_t^{\infty}\frac{\omega(s)}{s}\,\frac{ds}{s}
\leq C_\omega\,\frac{\omega(t)}{t}.
\end{equation}
The first inequality uses that $\omega$ is non-decreasing and the second is just \eqref{omega-cond:b}. 
Then, given $t_1\leq t_2$, we may write 
\begin{equation}\label{6gDDD.24}
\frac{\omega(t_2)}{t_2}\leq\int_{t_2}^{\infty}\frac{\omega(s)}{s}\,\frac{ds}{s}
\leq\int_{t_1}^{\infty}\frac{\omega(s)}{s}\,\frac{ds}{s}
\leq C_{\omega}\frac{\omega(t_1)}{t_1},
\end{equation} 
proving \eqref{omega-t-incre}. The doubling property in \eqref{item:wdoubling} follows at once 
from \eqref{omega-t-incre} by taking $t_2:=2t_1$ in \eqref{wnonincreasing}. Next, the claim in \eqref{wlimit} 
is justified by passing to limit $t\to\infty$ in the first inequality in \eqref{eq:fafr} 
and using Lebesgue's Dominated Convergence Theorem.

Turning our attention to \eqref{holderclosure}, fix an arbitrary $u\in\mathscr{C}^{\omega}(E,\mathbb{C}^M)$. 
As noted earlier, this membership ensures that $u$ is uniformly continuous, hence $u$ extends uniquely 
to a continuous function $v$ on $\overline{E}$. To show that $v$ belongs to 
$\dot{\mathscr{C}}^\omega(\overline{E},\mathbb{C}^M)$ pick two arbitrary distinct points $y,z\in\overline{E}$
and choose two sequences $\{y_k\}_{k\in{\mathbb{N}}}$, $\{z_k\}_{k\in{\mathbb{N}}}$ of points in $E$ 
such that $y_k\to x$ and $z_k\to z$ as $k\to\infty$. By discarding finitely many terms, there is no 
loss of generality in assuming that $|y_k-z_k|<2|y-z|$ for each $k\in{\mathbb{N}}$. 
Relying on the fact that $\omega$ is non-decreasing and 
\eqref{wdoubling}, we may then write 
\begin{multline}
|v(y)-v(z)|=\lim_{k\to\infty}|u(y_k)-u(z_k)|
\leq[u]_{\dot{\mathscr{C}}^{\omega}(E,\mathbb{C}^M)}\limsup_{k\to\infty}\omega(|y_k-z_k|) 
\\[4pt] 
\leq[u]_{\dot{\mathscr{C}}^{\omega}(E,\mathbb{C}^M)}\,\omega(2|y-z|) 
\leq 2C_\omega [u]_{\dot{\mathscr{C}}^{\omega}(E,\mathbb{C}^M)}\,\omega(|y-z|).
\end{multline}
From this, all claims in \eqref{holderclosure} follow, completing the proof of the lemma. 
\end{proof}

In the following lemma we treat $W$ defined in \eqref{omega-cond:WDef} as a growth 
function depending on the original $\omega$. 

\begin{lemma}\label{wlemma3}
Let $\omega$ be a growth function satisfying \eqref{omega-cond:a} and \eqref{omega-cond:b}, 
and let $W(t)$ be defined as in \eqref{omega-cond:WDef}. Then $W:(0,\infty)\to(0,\infty)$ is a growth function 
satisfying \eqref{omega-cond:b} with 
\begin{equation}\label{wW.aaa}
C_W\leq 1+(C_\omega)^2. 
\end{equation}
Moreover,
\begin{equation}\label{wW} 
\omega(t)\leq C_\omega\,W(t)\,\,\text{ for each }\,\,t\in(0,\infty). 
\end{equation}
\end{lemma}

\begin{proof}
By design, $W$ is a non-decreasing function and, thanks to Lebesgue's Dominated Convergence Theorem 
and \eqref{omega-cond:a} we have $W(t)\to 0$ as $t\to 0^{+}$. Also, on account of \eqref{wnonincreasing}, 
for each $t\in(0,\infty)$ we may write 
\begin{equation}\label{tfDDq.at}
\omega(t)=\int_0^t\frac{\omega(t)}{t}\,ds
\leq C_\omega\int_0^t\frac{\omega(s)}{s}\,ds=C_\omega W(t),
\end{equation}
proving \eqref{wW}. In turn, Fubini's Theorem, \eqref{omega-cond:b}, and \eqref{wW} permit us to estimate 
\begin{align}\label{tfDDq}
t\int_t^\infty\frac{W(s)}{s}\frac{ds}{s} 
&=t\int_t^\infty\left(\int_0^s\omega(\lambda)\frac{d\lambda}{\lambda}\right)\frac{ds}{s^2} 
\nonumber\\[4pt] 
&=t\int_0^t\left(\int_t^{\infty}\frac{ds}{s^2}\right)\omega(\lambda)\frac{d\lambda}{\lambda} 
+t\int_t^\infty\left(\int_\lambda^{\infty} \frac{ds}{s^2}\right)\omega(\lambda)\frac{d\lambda}{\lambda}
\nonumber\\[4pt] 
&=t\int_0^t\frac1{t}\omega(\lambda)\frac{d\lambda}{\lambda} 
+t\int_t^{\infty}\frac{\omega(\lambda)}{\lambda}\frac{d\lambda}{\lambda}
\nonumber\\[4pt] 
&\leq W(t)+C_\omega\,\omega(t)
\nonumber\\[4pt] 
&\leq\left(1+(C_\omega)^2\right)W(t),
\end{align}
for each $t\in(0,\infty)$. This shows that $W$ satisfies \eqref{omega-cond:b}
with constant $C_W\leq 1+(C_\omega)^2$.
\end{proof}

Moving on, for each given function $f\in L^1_{\loc}(\mathbb{R}^{n-1},\mathbb{C}^M)$ define the $L^p$-based 
mean oscillation of $f$ at a scale $r\in(0,\infty)$ as 
\begin{equation}\label{tfDDq.rD}
\osc_p(f;r):=\sup_{\substack{Q\subset\mathbb{R}^{n-1}\\ \ell(Q)\leq r}}
\bigg(\fint_Q|f(x')-f_Q|^p\,dx'\bigg)^{1/p}.
\end{equation}

The following lemma gathers some results from \cite[Lemmas 2.1 and 2.2]{BMO-MarMitMitMit16}.

\begin{lemma}\label{M4lem}
Let $f\in L^1_{\loc}(\mathbb{R}^{n-1},\mathbb{C}^M)$.

\begin{list}{\textup{(\theenumi)}}{\usecounter{enumi}\leftmargin=1cm \labelwidth=1cm \itemsep=0.2cm \topsep=.2cm \renewcommand{\theenumi}{\alph{enumi}}}

\item For every $p,q\in[1,\infty)$ there exists some finite $C=C(p,q,n)>1$ such that 
\begin{equation}\label{estimatea}
C^{-1}\osc_p(f;r)\leq\osc_q (f;r)\leq C\osc_p (f;r),\quad\forall\,r> 0.
\end{equation}

\item For every $\varepsilon>0$, 
\begin{equation}\label{estimateb}
\int_{1}^{\infty}\osc_1 (f;s)\frac{ds}{s^{1+\varepsilon}}<\infty
\,\Longrightarrow\,f\in L^1\left(\mathbb{R}^{n-1},\frac{dx'}{1+|x'|^{n-1+\varepsilon}}\right)^M.
\end{equation}
\end{list}
\end{lemma}

We augment Lemma~\ref{M4lem} with similar results involving generalized Morrey-Campanato spaces 
and generalized H\"{o}lder spaces.

\begin{lemma}\label{wlemma2}
Let $\omega$ be a growth function and fix $p\in[1,\infty)$. Then the following properties are valid. 

\begin{list}{\textup{(\theenumi)}}{\usecounter{enumi}\leftmargin=1cm \labelwidth=1cm \itemsep=0.2cm \topsep=.2cm \renewcommand{\theenumi}{\alph{enumi}}}

\item\label{osc-f-Ewp} 
If $f\in\mathscr{E}^{\omega,p}(\mathbb{R}^{n-1},\mathbb{C}^M)$, then
\begin{equation}\label{oscmorrey}
\osc_p(f;r)\leq\omega(r)\norm{f}_{\mathscr{E}^{\omega,p}(\mathbb{R}^{n-1},\mathbb{C}^M)}
\,\,\text{ for each }\,\,r\in(0,\infty).
\end{equation}

\item\label{inclusionitem} 
If $\omega$ satisfies \eqref{omega-cond:b}, then for each 
$f\in\mathscr{E}^{\omega,p}(\mathbb{R}^{n-1},\mathbb{C}^M)$ one has
\begin{equation}\label{easyinclusion} 
\norm{f}_{\mathscr{E}^{\omega,p}(\mathbb{R}^{n-1},\mathbb{C}^M)}
\leq\sqrt{n-1}\,C_\omega [f]_{\dot{\mathscr{C}}^{\omega}(\mathbb{R}^{n-1},\mathbb{C}^M)} 
\end{equation}
and 
\begin{equation}\label{inclusionE} 
\dot{\mathscr{C}}^{\omega}(\mathbb{R}^{n-1},\mathbb{C}^M)\subset 
\mathscr{E}^{\omega,p}(\mathbb{R}^{n-1},\mathbb{C}^M)\subset 
L^1\left(\mathbb{R}^{n-1},\frac{dx'}{1+|x'|^{n}}\right)^M. 
\end{equation}
\end{list}
\end{lemma}

\begin{proof}
Note that given any $f\in\mathscr{E}^{\omega,p}(\mathbb{R}^{n-1},\mathbb{C}^M)$ and $r>0$, 
based on \eqref{tfDDq.rD}, the fact that $\omega$ is non-decreasing, and \eqref{987tF} we may write 
\begin{align}\label{lemeq2}
\osc_p(f;r) &=\sup_{\substack{Q\subset\mathbb{R}^{n-1}\\ \ell(Q)\leq r}}
\omega(\ell(Q))\frac1{\omega(\ell(Q))}\bigg(\fint_Q|f(x')-f_Q|^p\,dx'\bigg)^{1/p}
\nonumber\\[6pt]
&\leq\omega(r)\norm{f}_{\mathscr{E}^{\omega,p}(\mathbb{R}^{n-1},\mathbb{C}^M)},
\end{align}
proving \eqref{osc-f-Ewp}. Consider next the claims in \eqref{inclusionitem}. 
Given any $f\in\dot{\mathscr{C}^{\omega}}(\mathbb{R}^{n-1},\mathbb{C}^M)$, a combination 
of \eqref{987tF}, \eqref{7tEEE}, and \eqref{wnonincreasing} yields 
\begin{align}\label{lemeq1}
\norm{f}_{\mathscr{E}^{\omega,p}(\mathbb{R}^{n-1},\mathbb{C}^M)}
&\leq\sup_{Q\subset\mathbb{R}^{n-1}}\left(\fint_Q\fint_Q\bigg(\frac{|f(x')-f(y')|}
{\omega(\ell(Q))}\bigg)^p\,dx'\,dy'\right)^{1/p} 
\nonumber\\[4pt]
&\leq\sup_{Q\subset\mathbb{R}^{n-1}}\frac{\omega(\sqrt{n-1}\ell(Q))}{\omega(\ell(Q))}\,
[f]_{\dot{\mathscr{C}^{\omega}}(\mathbb{R}^{n-1},\mathbb{C}^M)}
\nonumber\\[4pt]
&\leq\sqrt{n-1}C_\omega[f]_{\dot{\mathscr{C}^{\omega}}(\mathbb{R}^{n-1},\mathbb{C}^M)}.
\end{align}
This establishes \eqref{easyinclusion}, hence also the first inclusion in \eqref{inclusionE}. 
For the second inclusion in \eqref{inclusionE}, using Jensen's inequality, \eqref{oscmorrey}, 
and \eqref{omega-cond:b} we may write 
\begin{align}\label{tfDRV}
\int_1^{\infty}\osc_1(f;s)\frac{ds}{s^2} &\leq 
\norm{f}_{\mathscr{E}^{\omega,p}(\mathbb{R}^{n-1},\mathbb{C}^M)}\int_1^{\infty}\omega(s)\frac{ds}{s^2} 
\nonumber\\[6pt]
&\leq C_\omega\,\omega(1)\norm{f}_{\mathscr{E}^{\omega,p}(\mathbb{R}^{n-1},\mathbb{C}^M)}<\infty.
\end{align}
The desired inclusion now follows from this and \eqref{estimateb} with $\varepsilon:=1$.
\end{proof}

\section{Properties of Elliptic Systems and Their Solutions}\label{section:props-elliptic}

The following result is a particular case of more general interior estimates 
found in \cite[Theorem 11.9]{Mit13}.

\begin{theorem}\label{thmestimateder}
Let $L$ be a constant complex coefficient system as in \eqref{ars} satisfying \eqref{LegHad}. 
Then for every $p\in(0,\infty)$, $\lambda\in(0,1)$, and $m\in\mathbb{N}\cup\{0\}$ there exists 
a finite constant $C=C(L,p,m,\lambda,n)>0$ with the property that for every null-solution $u$ 
of $L$ in a ball $B(x,R)$, where $x\in\mathbb{R}^n$ and $R>0$, and every $r\in(0,R)$ one has
\begin{equation}\label{6gDDD.a4}
\sup_{z\in B(x,\lambda r)}|(\nabla^{m} u)(z)|\leq\frac{C}{r^m}\left(\fint_{B(x,r)}|u(x)|^p\,dx\right)^{1/p}.
\end{equation}
\end{theorem}

To proceed, introduce
\begin{align}\label{6gDDD.arf}
W^{1,2}_{\rm bdd}(\mathbb{R}^n_{+}) &:=\big\{u\in L^2_{\loc}(\mathbb{R}^n_{+}):\, 
u,\partial_j u\in L^2\big(\mathbb{R}^n_{+}\cap B(0,r)\big)
\nonumber\\[0pt]
&\hskip 0.80in
\,\text{ for each }\,\,j\in\{1,\dots,n\}\,\,\text{ and }\,\,r\in(0,\infty)\big\},
\end{align}
and define the Sobolev trace $\Tr$, whenever meaningful, as
\begin{equation}\label{defi-trace}
(\Tr\,u)(x'):=\lim_{r\to 0^{+}}\fint_{B((x',0),r)\cap\mathbb{R}^n_{+}}u(y)\,dy,
\qquad\,x'\in\mathbb{R}^{n-1}.
\end{equation}

The following result is taken from \cite[Corollary 2.4]{MazMitSha10}.

\begin{proposition}\label{uniqprop}
Let $L$ be a constant complex coefficient system as in \eqref{ars} satisfying \eqref{LegHad}, 
and suppose $u\in W^{1,2}_{\rm bdd}(\mathbb{R}^n_{+})$ satisfies $Lu=0$ in $\mathbb{R}^n_{+}$ and 
$\Tr\,u=0$ on $\mathbb{R}^{n-1}$. Then $u\in\mathscr{C}^{\infty}(\mathbb{R}^{n}_{+},\mathbb{C}^M)$ 
and there exists a finite constant $C>0$, independent of $u$, such that for each 
$x\in\overline{\mathbb{R}^n_{+}}$ and each $r>0$,
\begin{equation}\label{6gDDD.ar}
\sup_{\mathbb{R}^n_{+}\cap B(x,r)}|\nabla u|\leq\frac{C}{r}\sup_{\mathbb{R}^n_{+}\cap B(x,2r)}|u|.
\end{equation}
\end{proposition}

The following theorem is contained in \cite[Theorem 2.3 and Proposition 3.1]{BMO-MarMitMitMit16}.

\begin{theorem}\label{PoissonConvolution}
Suppose $L$ is a constant complex coefficient system as in \eqref{ars}, satisfying \eqref{LegHad}. 
Then the following statements are true. 

\begin{list}{\textup{(\theenumi)}}{\usecounter{enumi}\leftmargin=1cm \labelwidth=1cm \itemsep=0.2cm \topsep=.2cm \renewcommand{\theenumi}{\alph{enumi}}}

\item There exists a matrix-valued function 
$P^L=(P^L_{\alpha\beta})_{1\leq\alpha,\beta\leq M}:\mathbb{R}^{n-1}\to\mathbb{C}^{M\times M}$, 
called the Poisson kernel for $L$ in $\mathbb{R}^{n}_{+}$, such that 
$P^L\in{\mathscr{C}}^\infty(\mathbb{R}^{n-1})$, there exists some finite constant $C>0$ such that
\begin{equation}\label{PoissonDecay}
|P^L(x')|\leq\frac{C}{(1+|x'|^2)^{n/2}},\qquad\forall\,x'\in\mathbb{R}^{n-1},
\end{equation} 
and 
\begin{equation}\label{PoissonConvolutionId}
\int_{\mathbb{R}^{n-1}}P^L(x')\,dx'=I_{M\times M},
\end{equation} 
where $I_{M\times M}$ stands for the $M\times M$ identity matrix. 
Moreover, if for every $x'\in\mathbb{R}^{n-1}$ and $t>0$ one defines 
\begin{equation}\label{7hfDD}
K^L(x',t):=P_t^L(x'):= t^{1-n}P^L(x'/t),
\end{equation}
then $K^L\in{\mathscr{C}}^\infty\big(\overline{{\mathbb{R}}^n_{+}}\setminus B(0,\varepsilon)\big)$
for every $\varepsilon>0$ and the function $K^L=\big(K^L_{\alpha\beta}\big)_{1\leq\alpha,\beta\leq M}$
satisfies 
\begin{equation}\label{uahgab-UBVCX}
LK^L_{\cdot\beta}=0\,\,\text{ in }\,\,\mathbb{R}^{n}_{+}
\,\,\text{ for each }\,\,\beta\in\{1,\dots,M\},
\end{equation}
where $K^L_{\cdot\beta}:=\big(K^L_{\alpha\beta}\big)_{1\leq\alpha\leq M}$ is the $\beta$-th column in $K^L$.

\item\label{convo-poiss-sols} 
For each function $f=(f_\beta)_{1\le\beta\le M}\in L^1\big(\mathbb{R}^{n-1},\frac{dx'}{1+|x'|^n}\big)^M$
define, with $P^L$ as above, 
\begin{equation}\label{ConvDef} 
u(x',t):=(P^L_t\ast f)(x'),\qquad\forall\,(x',t)\in\mathbb{R}^{n}_{+}. 
\end{equation} 
Then $u$ is meaningfully defined, via an absolutely convergent integral, and satisfies
\begin{equation}\label{exist:u2****}
u\in\mathscr{C}^\infty(\mathbb{R}^n_{+},{\mathbb{C}}^M),\quad\,\, 
Lu=0\,\,\text{ in }\,\,\mathbb{R}^{n}_{+},\quad\,\,
\restr{u}{\partial\mathbb{R}^n_{+}}^{{}^{\rm nt.lim}}=f\text{ a.e. on }\,\mathbb{R}^{n-1}.
\end{equation} 
Furthermore, there exists a finite constant $C>0$ such that 
\begin{equation}\label{ConvDer}
|(\nabla u)(x',t)|\leq\frac{C}{t}\int_1^{\infty}\osc_1(f;st)\frac{ds}{s^{2}},
\qquad\forall\,(x',t)\in\mathbb{R}^{n}_{+},
\end{equation}
and, for each cube $Q\subset\mathbb{R}^{n-1}$, 
\begin{equation}\label{PoissonConvolutioneq}
\left(\int_0^{\ell(Q)}\fint_Q|(\nabla u)(x',t)|^2\,t\,dx'\,dt\right)^{1/2}
\leq C\int_1^{\infty}\osc_1(f; s\ell(Q))\frac{ds}{s^2}.
\end{equation}
\end{list}
\end{theorem}

Our next proposition contains a number of a priori estimates comparing 
$\norm{u}_{**}^{(\omega,q)}$, corresponding to different values of $q$, for solutions 
of $Lu=0$ in ${\mathbb{R}}^n_{+}$. To set the stage, we first state some simple estimates 
which are true for any function $u\in\mathscr{C}^{1}(\mathbb{R}^n_{+},\mathbb{C}^M)$:
\begin{equation}\label{uqeq:1}
\norm{u}_{**}^{(\omega,p)}\leq\norm{u}_{**}^{(\omega,q)}
\leq C\norm{u}_{**}^{(\omega,\exp)},\qquad 0<p\le q<\infty,
\end{equation}
where $C=C(q)\geq 1$. Indeed, the first estimate follows at once from Jensen's inequality. 
The second estimate is a consequence of the fact that $t^{\max \{1 , q \} }\leq C(e^{t}-1)$ (with $C>0$ depending on $\max \{ 1,q \}$) for each $t\in(0,\infty)$ 
and the definition of $\|\cdot\|_{\exp L,Q}$ (cf. \eqref{ygFF}).

\begin{proposition}\label{prop:proper-sols}
Let $L$ be a constant complex coefficient system as in \eqref{ars} satisfying the strong ellipticity 
condition \eqref{LegHad}, and let $u\in\mathscr{C}^{\infty}(\mathbb{R}^{n}_{+},\mathbb{C}^M)$ be such 
that $Lu=0$ in $\mathbb{R}^n_+$. Then the following statements hold.

\begin{list}{\textup{(\theenumi)}}{\usecounter{enumi}\leftmargin=1cm \labelwidth=1cm \itemsep=0.2cm \topsep=.2cm \renewcommand{\theenumi}{\alph{enumi}}}

\item\label{int-Whitney} 
For every $q\in(0,\infty)$ there there exists a finite constant $C=C(L,n,q)\geq 1$ such that for 
each $(x',t)\in\mathbb{R}^{n}_{+}$ one has
\begin{equation}\label{eq:inter-est}
t\,|(\nabla u)(x',t)|\leq C
\bigg(\fint_{|x'-y'|<\frac{t}{2}}\bigg(\int_{t/2}^{3t/2}|(\nabla u)(y',s)|^2\,s\,ds\bigg)^{q/2}\,dy'\bigg)^{1/q}.
\end{equation}

\item\label{con-le-carl} 
There exists a finite constant $C=C(L,n)\geq 1$ such that for each cube $Q\subset\mathbb{R}^{n-1}$
and each $x'\in\mathbb{R}^{n-1}$ one has
\begin{align}\label{eq:conical-carl-q:1}
\bigg(\int_0^{\ell(Q)}|(\nabla u)(x',t)|^2\,t\,dt\bigg)^{1/2}
\leq C\,\bigg(\int_{0}^{2\ell(Q)}\int_{|x'-y'|<s}|(\nabla u)(y',s)\,s|^2\,dy'\frac{ds}{s^n}\bigg)^{1/2}.
\end{align}
Furthermore, whenever $2\leq q<\infty$ there exists a finite constant $C=C(L,n,q)\geq 1$ such that 
for each cube $Q\subset\mathbb{R}^{n-1}$ and each $x'\in\mathbb{R}^{n-1}$ one has
\begin{multline}\label{eq:conical-carl-q:2}
\bigg(\fint_{Q}\bigg(\int_{0}^{\ell(Q)}\int_{|x'-y'|<s}
|(\nabla u)(y',s)\,s|^2\,dy'\frac{ds}{s^n}\bigg)^{q/2}\,dx'\bigg)^{1/q}
\\[4pt]
\leq C\,\bigg(\fint_{3Q}\bigg(\int_0^{3\ell(Q)}|(\nabla u)(x',t)|^2\,t\,dt\bigg)^{q/2}\,dx'\bigg)^{1/q}.
\end{multline}

\item\label{sup-Comega} 
There exists a finite constant $C=C(L,n)\geq 1$ such that for each growth function $\omega$ one has 
\begin{equation}\label{lemsmalleq}
\norm{u}_{**}^{(\omega,\infty)}\leq C[u]_{\dot{\mathscr{C}^{\omega}}(\mathbb{R}^{n}_{+},\mathbb{C}^M)}.
\end{equation}

\item\label{sup-q-car} 
For every $q\in(0,\infty)$ there exists a finite constant $C=C(L,n,q)\geq 1$ such that 
for each growth function $\omega$ satisfying \eqref{omega-cond:b} one has 
\begin{equation}\label{uqeq-infty}
\norm{u}_{**}^{(\omega,\infty)}\leq C\,C_\omega\norm{u}_{**}^{(\omega,q)}.
\end{equation}

\item\label{comp-car-p-q-2} 
There exists a finite constant $C=C(L,n)\geq 1$ such that for each growth function 
$\omega$ satisfying \eqref{omega-cond:b} one has 
\begin{equation}\label{uqeq}
\norm{u}_{**}^{(\omega,\exp)}\leq C(C_\omega)^2\norm{u}_{**}^{(\omega,2)}.
\end{equation}

\item\label{CW-car-q} 
Let $\omega$ be a growth function satisfying \eqref{omega-cond:a} as well as \eqref{omega-cond:b}, 
and define $W(t)$ as in \eqref{omega-cond:WDef}. Then
\begin{equation}\label{mainthmBMOeq4}
[u]_{\dot{\mathscr{C}}^{W}(\mathbb{R}^{n}_{+},\mathbb{C}^M)} 
\leq C_\omega (2+C_\omega)\norm{u}_{**}^{(\omega,\infty)},
\end{equation}
and, if the latter quantity is finite, $u\in\dot{\mathscr{C}}^{W}(\overline{\mathbb{R}^n_{+}},\mathbb{C}^M)$ 
in the sense of Lemma~\ref{wlemma}\eqref{holderclosure}.

\item\label{car-q-sup} 
Let $\omega$ be a growth function satisfying 
\begin{equation}\label{omega-cond:second}
\int_0^{t}\omega(s)\frac{ds}{s}\leq C_\omega'\,\omega(t),\qquad\forall\,t\in(0,\infty),
\end{equation}
for some finite constant $C_\omega'>1$. Then
\begin{equation}\label{uqeq-infty:converse}
\norm{u}_{**}^{(\omega,\exp)}\le (C_\omega')^{1/2}\norm{u}_{**}^{(\omega,\infty)}.
\end{equation}

\item\label{all-comp} 
Let $\omega$ be a growth function satisfying \eqref{omega-cond:main}. Then for every $q\in(0,\infty)$ 
\begin{equation}\label{relation-all-norms:1}
\norm{u}_{**}^{(\omega,q)}\approx\norm{u}_{**}^{(\omega,\exp)}\approx\norm{u}_{**}^{(\omega,\infty)}
\approx [u]_{\dot{\mathscr{C}}^{w}(\mathbb{R}^{n}_{+},\mathbb{C}^M)} 
\end{equation}
where the implicit constants depend only on $L$, $n$, $q$, and the constant $C_0$ 
in \eqref{omega-cond:main}. In particular, if $\norm{u}_{**}^{(\omega,q)}<\infty$ 
for some $q\in(0,\infty]$, or $\norm{u}_{**}^{(\omega,\exp)}<\infty$, then 
$u\in\dot{\mathscr{C}}^{\omega}(\overline{\mathbb{R}^n_{+}},\mathbb{C}^M)$ 
in the sense of Lemma~\ref{wlemma}\eqref{holderclosure}.
\end{list}
\end{proposition}

\begin{proof}
We start by proving \eqref{int-Whitney}. Fix $(x',t)\in\mathbb{R}^{n}_{+}$ and let $Q_{x',t}$
be the cube in $\mathbb{R}^{n-1}$ centered at $x'$ with side-length $t$. Then from 
Theorem~\ref{thmestimateder} (presently used with $m:=0$ and $p:=\min\{q,2\}$) and Jensen's inequality 
we obtain 
\begin{align}\label{step2}
|(\nabla u)(x',t)| &\leq C\left(\fint_{|(y',s)-(x',t)|<\frac{t}{2}}|(\nabla u)(y',s)|^p\,dy'\,ds\right)^{1/p}
\nonumber\\[4pt] 
&\leq C\left(\fint_{|x'-y'|<\frac{t}{2}}\bigg(\fint_{(t/2,3t/2)}|(\nabla u)(y',s)|^2\,ds\bigg)^{p/2}\,dy'\right)^{1/p}
\nonumber\\[4pt] 
&\leq C\bigg(\fint_{|x'-y'|<\frac{t}{2}}\bigg(\fint_{(t/2,3t/2)}|(\nabla u)(y',s)|^2\,ds\bigg)^{q/2}\,dy'\bigg)^{1/q}
\nonumber\\[4pt] 
&=Ct^{-1}\bigg(\fint_{|x'-y'|<\frac{t}{2}}\bigg(\int_{t/2}^{3t/2}|(\nabla u)(y',s)|^2\,s\,ds\bigg)^{q/2} dy'\bigg)^{1/q},
\end{align}
proving \eqref{eq:inter-est}. Turning our attention to \eqref{con-le-carl}, fix 
a cube $Q\subset\mathbb{R}^{n-1}$ along with a point $x'\in\mathbb{R}^{n-1}$. 
First, integrating \eqref{eq:inter-est} written for $q:=2$ yields 
\begin{align}\label{6gDDD.atrf}
\int_{0}^{\ell(Q)}|(\nabla u)(x',t)|^2\,t\,dt 
&\leq C\int_{0}^{\ell(Q)}\frac1{t^{n+1}}\int_{t/2}^{3t/2}\int_{|x'-y'|<s}|(\nabla u)(y',s)|^2\,s\,dy'\,ds\,t\,dt
\nonumber\\[4pt] 
&\leq C\int_{0}^{2\ell(Q)}\int_{|x'-y'|<s}|(\nabla u)(y',s)|^2\int_{2s/3}^{2s}\,t^{-n}\,dt\,dy'\,s\,ds
\nonumber\\[4pt] 
&=C\int_{0}^{2\ell(Q)}\int_{|x'-y'|<s}|(\nabla u)(y',s)\,s|^2\,dy'\,\frac{ds}{s^n},
\end{align}
and this readily leads to the estimate in \eqref{eq:conical-carl-q:1}. To justify \eqref{eq:conical-carl-q:2}, 
observe that for each nonnegative function $h\in L^1_{\rm loc}({\mathbb{R}}^{n-1})$ we have
\begin{align}\label{eq:Con-vert}
&\fint_{Q}\bigg(\int_{0}^{\ell(Q)}\int_{|x'-y'|<s}|(\nabla u)(y',s)\,s|^2\,dy'\frac{ds}{s^n}\bigg)h(x')\,dx'
\nonumber\\[4pt]
&\qquad\leq 3^n\fint_{3Q}\int_{0}^{\ell(Q)}\bigg(\frac1{s^{n-1}}\int_{|y'-x'|<s}h(x')\,dx'\bigg) 
|(\nabla u)(y',s)|^2\,s\,ds\,dy'
\nonumber\\[4pt]
&\qquad\leq C_n\fint_{3Q}\bigg(\int_{0}^{3\ell(Q)}|(\nabla u)(y',s)|^2\,s\,ds\bigg)(Mh)(x')\,dx',
\end{align}
where $M$ is the Hardy-Littlewood maximal operator in ${\mathbb{R}}^{n-1}$. 
Note that if $q=2$ then \eqref{eq:Con-vert} gives at once \eqref{eq:conical-carl-q:2}
by taking $h=1$ in $Q$ and using that $Mh\leq 1$. On the other hand, if $q>2$, we  
impose the normalization condition $\|h\|_{L^{(q/2)'}(Q,dx'/|Q|)}=1$ and 
then rely on \eqref{eq:Con-vert} and H\"older's inequality to write
\begin{align}\label{eq:Con-vert-q>2}
&\fint_{Q}\bigg(\int_{0}^{\ell(Q)}\int_{|x'-y'|<s}|(\nabla u)(y',s)\,s|^2\,dy'\frac{ds}{s^n}\bigg)h(x')\,dx'
\nonumber\\[4pt]
&\qquad\leq C_n\bigg(\fint_{3Q}\bigg(\int_{0}^{3\ell(Q)}|(\nabla u)(y',s)|^2\,s\,ds\bigg)^{q/2}\,dx'\bigg)^{2/q}
\,\|Mh\|_{L^{(q/2)'}(Q,dx'/|Q|)}
\nonumber\\[4pt]
&\qquad\leq C\bigg(\fint_{3Q}\bigg(\int_{0}^{3\ell(Q)}|(\nabla u)(y',s)|^2\,s\,ds\bigg)^{q/2}\,dx'\bigg)^{2/q},
\end{align}
bearing in mind that $M$ is bounded in $L^{(q/2)'}(\mathbb{R}^{n-1})$, given that $q>2$. 
Taking now the supremum over all such functions $h$ yields \eqref{eq:conical-carl-q:2}
on account of Riesz' duality theorem. 

As regards \eqref{sup-Comega}, fix $(x',t)\in\mathbb{R}^n_{+}$ and use 
Theorem~\ref{thmestimateder} together with the fact that $\omega$ is a non-decreasing function
to write 
\begin{align}\label{BMOuniqeq1}
|(\nabla u)(x',t)| &=\big|\nabla(u(\cdot)-u(x',t))(x',t)\big| 
\nonumber\\[4pt]
&\leq\frac{C}{t}\fint_{|(y',s)-(x',t)|<t/2}|u(y',s)-u(x',t)|\,dy'\,ds
\nonumber\\[4pt]
&\leq C[u]_{\dot{\mathscr{C}^{\omega}}(\mathbb{R}^{n}_{+},\mathbb{C}^M)}\frac{\omega(t)}{t}.
\end{align} 
In view of \eqref{yRRFF}, this readily establishes \eqref{lemsmalleq}.
 
The claim in \eqref{sup-q-car} is proved by combining \eqref{eq:inter-est} and \eqref{wdoubling}, which permit us 
to estimate (recall that $Q_{x',t}$ denotes the cube in $\mathbb{R}^{n-1}$ centered at $x'$ with side-length $t$)
\begin{align}\label{step4dedawe}
\norm{u}_{**}^{(\omega,\infty)}
&\leq C\,\sup_{(x',t)\in\mathbb{R}^{n}_{+}}\frac{1}{\omega(t)}
\bigg(\fint_{(3/2) Q_{x',t}}\bigg(\int_{0}^{3t/2}|(\nabla u)(y',s)|^2\,s\,ds\bigg)^{q/2}\,dy'\bigg)^{1/q}
\nonumber\\[4pt] 
&\leq C\,C_\omega\norm{u}_{**}^{(\omega,q)}.
\end{align}

Going further, consider the claim in \eqref{comp-car-p-q-2}. For starters, observe that the 
convexity of the function $t\mapsto e^{t}-1$ readily implies that $2^{n-1}(e^{t}-1)\leq e^{2^{n-1}t}-1$ 
for every $t>0$ which, in view of \eqref{eq:Lux-norm}, allows us to write 
\begin{equation}\label{eq:exp-Q-2Q}
\|f\|_{\exp L,Q}\leq 2^{n-1}\,\|f\|_{\exp L,2Q}
\end{equation}
for each cube $Q$ in ${\mathbb{R}}^{n-1}$ and each Lebesgue measurable function $f$ on $Q$.

Turning to the proof of \eqref{uqeq} in earnest, by homogeneity we may assume that 
$\norm{u}^{(\omega,2)}_{**}=1$ to begin with. We are going to use Lemma~\ref{lem:appendix}. 
As a prelude, define 
\begin{equation}\label{6agRD4r7}
F(y',s):=|(\nabla u)(y',s)\,s|,\qquad\forall\,(y',s)\in\mathbb{R}^n_{+},
\end{equation}
and, for each cube $Q$ in ${\mathbb{R}}^{n-1}$ and each threshold $N\in(0,\infty)$, consider the set 
\begin{equation}\label{6gDDD.a4r9jh}
E_{N,Q}:=\bigg\{x'\in Q:\,\frac{1}{\omega(\ell(Q))}\bigg(\int_0^{\ell(Q)}
\int_{|x'-y'|<\kappa s}|F(y',s)|^2\,dy'\frac{ds}{s^n}\bigg)^{1/2}>N\bigg\}.
\end{equation} 
where $\kappa:=1+2\sqrt{n-1}$. Denoting $Q^{*}:=(2\kappa+1)Q=(3+4\sqrt{n-1})Q$, then
using Chebytcheff's inequality, and \eqref{wnonincreasing}, for each cube $Q$ in ${\mathbb{R}}^{n-1}$ 
and each $N>0$ we may write 
\begin{align}\label{6gDDD.a433}
|E_{N,Q}| &\leq\frac{1}{N^2}\frac{1}{\omega(\ell(Q))^2}
\int_Q\int_0^{\ell(Q)}\int_{|x'-y'|<\kappa s}|F(y',s)|^2\,dy'\frac{ds}{s^n}\,dx'
\nonumber \\[4pt]
&\leq\frac{1}{N^2}\frac{1}{\omega(\ell(Q))^2}
\int_{Q^{*}}\int_0^{\ell(Q)}\Big(\int_{|y'-x'|<\kappa s}\,dx'\Big)\,|F(y',s)|^2\frac{ds}{s^n}\,dy'
\nonumber\\[4pt]
&\leq C\frac{1}{N^2}\frac{1}{\omega(\ell(Q))^2}\int_{Q^{*}}\int_0^{\ell(Q^{*})} 
|(\nabla u)(y',s)\,s|^2\frac{ds}{s}\,dy'
\nonumber\\[4pt] 
&\leq C\frac{|Q^{*}|}{N^2}\frac{\omega(\ell(Q^{*}))^2}{\omega(\ell(Q))^2}
\big(\norm{u}^{(\omega,2)}_{**}\big)^2
=C\frac{|Q^{*}|}{N^2}\Big[\frac{\omega(\ell(Q^{*}))}{\omega(\ell(Q))}\Big]^2
\nonumber\\[4pt] 
&\leq C_0(C_{\omega})^2\frac{1}{N^2}|Q|,
\end{align} 
for some finite constant $C_0>0$. Therefore, taking $N:=\sqrt{2C_0}C_{\omega}>0$, we conclude that
\begin{equation}\label{eq:lemmaapplies}
|E_{N,Q}|\leq\frac{1}{2}|Q|.
\end{equation}  
This allows us to invoke Lemma~\ref{lem:appendix} with $\varphi:=\omega$, which together with 
\eqref{eq:conical-carl-q:1}, \eqref{wdoubling}, and \eqref{eq:exp-Q-2Q},  gives
\begin{align}\label{iatfa} 
\norm{u}^{(\omega,\exp)}_{**} 
&\leq C\sup_{Q\subset\mathbb{R}^{n-1}}\frac{1}{\omega(\ell(Q))} 
\bigg\|\bigg(\int_{0}^{\ell(2Q)}\int_{|\,\cdot\,-y'|<s}|F(y',s)|^2\,dy'\frac{ds}{s^n}\bigg)^{1/2}\bigg\|_{\exp L,Q}
\nonumber\\[4pt] 
&\leq C(C_{\omega})^{2}.
\end{align}
This completes the proof of \eqref{comp-car-p-q-2}.

Turning our attention to \eqref{CW-car-q}, fix $x=(x',t)$ and $y=(y',s)$ in $\mathbb{R}^{n}_{+}$, 
and abbreviate $r:=|x-y|$. Then,
\begin{align}\label{step4eq1}
\frac{|u(x)-u(y)|}{W(|x-y|)} &\leq\frac{1}{W(r)}|u(x',t)-u(x',t+r)|+\frac1{W(r)}|u(x',t+r)-u(y',s+r)|
\nonumber\\[4pt] 
&\quad+\frac1{W(r)}|u(y',s+r)-u(y',s)|
\nonumber\\[4pt] 
&=:I+II+III.
\end{align}
To bound $I$, we use the Fundamental Theorem of Calculus, \eqref{wnonincreasing}, and 
\eqref{omega-cond:WDef} and obtain 
\begin{multline}\label{step4eq2}
I=\frac1{W(r)}\left|\int_0^r(\partial_n u)(x',t+\xi)\,d\xi\right| 
\leq\norm{u}_{**}^{(\omega,\infty)}\frac{1}{W(r)}\int_0^r\frac{\omega(t+\xi)}{t+\xi}\,d\xi 
\\[4pt]
\leq C_\omega\norm{u}_{**}^{(\omega,\infty)}\frac{1}{W(r)}\int_0^r\frac{\omega(\xi)}{\xi}\,d\xi 
=C_\omega\norm{u}_{**}^{(\omega,\infty)}.
\end{multline}
Note that $III$ is bounded analogously replacing $x'$ by $y'$ and $t$ by $s$. For $II$, 
we use again the Fundamental Theorem of Calculus, together with \eqref{wnonincreasing} and \eqref{wW},
to write 
\begin{align}\label{step4eq3}
II &=\frac{1}{W(r)}\left|\int_0^1\frac{d}{d\theta}[u(\theta(x',t+r)+(1-\theta)(y',s+r))]\,d\theta\right|
\nonumber\\[4pt] 
&=\frac{1}{W(r)}\left|\int_0^1(x'-y',t-s)\cdot(\nabla u)\big(\theta(x',t+r)+(1-\theta)(y',s+r)\big)\,d\theta\right|
\nonumber\\[4pt] 
&\leq\norm{u}_{**}^{(\omega,\infty)}\frac{r}{W(r)}\int_0^1
\frac{\omega\big((1-\theta)s+\theta t+r\big)}{(1-\theta)s+\theta t+r}\,d\theta
\nonumber\\[4pt] 
&\leq C_\omega\norm{u}_{**}^{(\omega,\infty)}\frac{r}{W(r)}\int_0^1\frac{\omega(r)}{r}\,d\theta
\nonumber\\[4pt] 
&\leq(C_\omega)^2\norm{u}_{**}^{(\omega,\infty)}.
\end{align}
As $x$ and $y$ were chosen arbitrarily, \eqref{step4eq1}, \eqref{step4eq2}, 
and \eqref{step4eq3} collectively justify \eqref{mainthmBMOeq4}.

To justify \eqref{car-q-sup}, observe that since $\omega$ is non-decreasing and satisfies 
\eqref{omega-cond:second} we may write 
\begin{align}\label{eq:76ergtr}
\bigg(\fint_Q\bigg(\int_0^{\ell(Q)}|(\nabla u)(x',t)|^2\,t\,dt\bigg)^{q/2}\,dx'\bigg)^{1/q}
&\leq\bigg(\int_0^{\ell(Q)}\omega(t)^2\frac{dt}{t}\bigg)^{1/2}\norm{u}_{**}^{(\omega,\infty)}
\nonumber\\[4pt]
&\leq(C_\omega')^{1/2}\omega(\ell(Q))\norm{u}_{**}^{(\omega,\infty)},
\end{align}
which readily leads to the desired inequality. 

As regards \eqref{all-comp}, the idea is to combine \eqref{uqeq:1}, \eqref{car-q-sup}, 
and \eqref{sup-q-car} for the first three equivalences. In concert, \eqref{sup-Comega}, 
the fact that \eqref{omega-cond:main} gives $W\leq C_0\,\omega$, and \eqref{CW-car-q}
also give the last equivalence in \eqref{all-comp}. The proof of Proposition~\ref{prop:proper-sols}
is therefore complete.
\end{proof}

\section{Existence Results}\label{section:Existence}

In this section we develop the main tools used to establish the existence of solutions for the boundary 
value problems formulated in the statement of Theorem~\ref{mainthmBMO}. We start with the 
generalized H\"older Dirichlet problem.

\begin{proposition}\label{propfholder}
Let $L$ be a constant complex coefficient system as in \eqref{ars} satisfying the strong ellipticity 
condition formulated in \eqref{LegHad}, and let $\omega$ be a growth function satisfying \eqref{omega-cond:b}. 
Given $f\in\dot{\mathscr{C}}^{\omega}(\mathbb{R}^{n-1},\mathbb{C}^M)$, define 
$u(x',t):=(P_t^L\ast f)(x')$ for every $(x',t)\in\mathbb{R}^{n}_{+}$. 
Then $u$ is meaningfully defined via an absolutely convergent integral and satisfies
\begin{equation}\label{exist:u2****-Holder-omega}
u\in\mathscr{C}^\infty(\mathbb{R}^n_{+},{\mathbb{C}}^M),\quad\,\,
Lu=0\,\,\mbox{ in }\,\,\mathbb{R}^{n}_{+},\quad\,\,
\restr{u}{\partial\mathbb{R}^n_{+}}^{{}^{\rm nt.lim}}=f\text{ a.e. on }\,\,\mathbb{R}^{n-1}.
\end{equation} 
Moreover, there exists a finite constant $C=C(L,n)>0$ such that 
\begin{equation}\label{fholdereq}
[u]_{\dot{\mathscr{C}}^{\omega}(\mathbb{R}^{n}_{+},\mathbb{C}^M)}
\leq C\,C_\omega(1+C_\omega)[f]_{\dot{\mathscr{C}}^{\omega}(\mathbb{R}^{n-1},\mathbb{C}^M)},
\end{equation} 
and $u\in\dot{\mathscr{C}}^{\omega}(\overline{\mathbb{R}^n_{+}},\mathbb{C}^M)$ 
with $\restr{u}{\partial\mathbb{R}^n_{+}}=f$. 
\end{proposition}

\begin{proof}
Let $f\in\dot{\mathscr{C}}^{\omega}(\mathbb{R}^{n-1},\mathbb{C}^M)$ and define 
$u(x',t):=(P_t^L\ast f)(x')$ for every $(x',t)\in\mathbb{R}^{n}_{+}$. By \eqref{inclusionE} and 
Theorem~\ref{PoissonConvolution}\eqref{convo-poiss-sols}, $u$ satisfies all 
properties listed in \eqref{exist:u2****-Holder-omega}. To prove the estimate in 
\eqref{fholdereq}, we first notice that for any $(x',t)\in\mathbb{R}_+^{n}$, 
we can write
\begin{align}\label{iatuyhG}
(P_t^L\ast f)(x') &=\int_{\mathbb{R}^{n-1}} P_t^L(x'-y')f(y')\,dy' 
=\int_{\mathbb{R}^{n-1}}t^{1-n}P^L\left(\frac{x'-y'}{t}\right)f(y')\,dy'
\nonumber\\[4pt] 
&=\int_{\mathbb{R}^{n-1}}P^L(z')f(x'-tz')\,dz'.
\end{align} 
Fix now $x=(x',t)$ and $y=(y',s)$ arbitrary in $\mathbb{R}^{n}_{+}$, and set $r:=|x-y|$. 
By \eqref{PoissonDecay} and the fact that $\omega$ is non-decreasing we obtain
\begin{align}\label{fholdereq-1}
|u(x',t)-u(y',s)|&=|(P^L_t\ast f)(x')-(P^L_s\ast f)(y')|
\nonumber\\[4pt] 
&\leq C\int_{\mathbb{R}^{n-1}}\frac1{(1+|z'|^2)^{n/2}}\, |f(x'-tz')-f(y'-sz')|\, dz'
\nonumber\\[4pt] 
&\leq C[f]_{\dot{\mathscr{C}}^{\omega}(\mathbb{R}^{n-1},\mathbb{C}^M)}
\int_{\mathbb{R}^{n-1}}\frac1{(1+|z'|^2)^{n/2}}\,\omega((1+|z'|)r)\,dz'
\nonumber\\[4pt] 
&\leq C [f]_{\dot{\mathscr{C}}^{\omega}(\mathbb{R}^{n-1},\mathbb{C}^M)}
\int_0^{\infty}\frac{1}{(1+\lambda^2)^{n/2}}\,\omega\big((1+\lambda)r\big)\,
\lambda^{n-1}\frac{d\lambda}{\lambda} 
\nonumber\\[4pt] 
&\leq C[f]_{\dot{\mathscr{C}}^{\omega}(\mathbb{R}^{n-1},\mathbb{C}^M)} 
\bigg(\int_0^1\omega(2 r)\,\lambda^{n-1}\frac{d\lambda}{\lambda}
+\int_1^{\infty}\frac{\omega(2\lambda r)}{\lambda}\frac{d\lambda}{\lambda}\bigg)
\nonumber\\[4pt] 
&=C[f]_{\dot{\mathscr{C}}^{\omega}(\mathbb{R}^{n-1},\mathbb{C}^M)} 
\bigg(\omega(2r)+2r\int_{2r}^{\infty}\frac{\omega(\lambda)}{\lambda}\frac{d\lambda}{\lambda}\bigg)
\nonumber\\[4pt] 
&\leq C\,C_\omega(1+C_\omega)\omega(r)[f]_{\dot{\mathscr{C}}^{\omega}(\mathbb{R}^{n-1},\mathbb{C}^M)}, 
\end{align} 
where in the last inequality we have used \eqref{omega-cond:b} and \eqref{wdoubling}. Hence, 
\eqref{fholdereq} holds. In particular, $u\in\dot{\mathscr{C}}^{\omega}(\overline{\mathbb{R}^n_{+}},\mathbb{C}^M)$ 
by Lemma~\ref{wlemma}\eqref{holderclosure}. This and the fact that 
$\restr{u}{\partial\mathbb{R}^n_{+}}^{{}^{\rm nt.lim}}=f$ a.e. in $\mathbb{R}^{n-1}$ with 
$f\in\dot{\mathscr{C}}^{\omega}(\mathbb{R}^{n-1},\mathbb{C}^M)$ then prove that, indeed,  
$\restr{u}{\partial\mathbb{R}^n_{+}}=f$.
\end{proof}

The result below is the main tool in the proof of existence of solutions for the generalized Morrey-Campanato 
Dirichlet problem.

\begin{proposition}\label{propstep1}
Let $L$ be a constant complex coefficient system as in \eqref{ars} satisfying the strong ellipticity 
condition stated in \eqref{LegHad}, and let $\omega$ be a growth function satisfying \eqref{omega-cond:b}.
Given $1\leq p<\infty$, let $f\in\mathscr{E}^{\omega,p}(\mathbb{R}^{n-1},\mathbb{C}^M)$ and define 
$u(x',t):=(P_t^L\ast f)(x')$ for every $(x',t)\in\mathbb{R}^{n}_{+}$. Then $u$ is meaningfully defined 
via an absolutely convergent integral and satisfies
\begin{equation}\label{exist:u2****-Morrey-campanato}
u\in\mathscr{C}^\infty(\mathbb{R}^n_{+},{\mathbb{C}}^M),\quad\,\, 
Lu=0\,\,\mbox{ in }\,\,\mathbb{R}^{n}_{+},\quad\,\,
\restr{u}{\partial\mathbb{R}^n_{+}}^{{}^{\rm nt.lim}}=f\text{ a.e. on }\,\,\mathbb{R}^{n-1}.
\end{equation} 
Moreover, for every $q\in(0,\infty]$ there exists a finite constant $C=C(L,n,p,q)>0$ such that  
\begin{equation}\label{step1}
\norm{u}_{**}^{(\omega,q)}\leq C(C_\omega)^{4}\norm{f}_{\mathscr{E}^{\omega,p}(\mathbb{R}^{n-1},\mathbb{C}^M)}.
\end{equation}
Furthermore, the same is true if $\norm{\cdot}_{**}^{(\omega,q)}$ is replaced
by $\norm{\cdot}_{**}^{(\omega,\exp)}$.
\end{proposition}

\begin{proof}
Given $f\in\mathscr{E}^{\omega,p}(\mathbb{R}^{n-1},\mathbb{C}^M)$, if $u(x',t):=(P_t^L\ast f)(x')$ 
for every $(x',t)\in\mathbb{R}^{n}_{+}$, from \eqref{inclusionE} and 
Theorem~\ref{PoissonConvolution}\eqref{convo-poiss-sols} we see that 
$u$ satisfies all properties listed in \eqref{exist:u2****-Morrey-campanato}.

Next, having fixed an arbitrary exponent $q\in(0,\infty)$, based on Proposition~\ref{prop:proper-sols}\eqref{sup-q-car}, 
\eqref{uqeq:1}, Proposition~\ref{prop:proper-sols}\eqref{comp-car-p-q-2}, \eqref{PoissonConvolutioneq}, 
\eqref{estimatea}, \eqref{oscmorrey} and \eqref{omega-cond:b} we may write 
\begin{align}\label{pa6ff}
\norm{u}_{**}^{(\omega,\infty)} &\leq C\,C_\omega\norm{u}_{**}^{(\omega,q)}  
\leq C\,C_\omega\norm{u}_{**}^{(\omega,\exp)}  
\nonumber\\[4pt]
&\leq C(C_\omega)^3\norm{u}_{**}^{(\omega,2)}  
\leq C(C_\omega)^3\sup_{t>0} \frac1{\omega(t)}\int_{1}^{\infty}\osc_1(f,st)\frac{ds}{s^2} 
\nonumber\\[4pt]
&=C(C_\omega)^3\norm{f}_{\mathscr{E}^{\omega,p}(\mathbb{R}^{n-1},\mathbb{C}^M)}
\sup_{t>0}\frac{t}{\omega(t)}\int_t^{\infty}\omega(s)\frac{ds}{s^2} 
\nonumber\\[4pt]
&\leq C(C_\omega)^{4}\norm{f}_{\mathscr{E}^{\omega,p}(\mathbb{R}^{n-1},\mathbb{C}^M)},
\end{align} 
which proves \eqref{step1} and the corresponding estimate for $\norm{u}_{**}^{(\omega,\exp)}$.
\end{proof}

\section{A Fatou-Type Result and Uniqueness of Solutions}\label{section:Fatou}

We shall now prove a Fatou-type result which is going to be the main ingredient in establishing 
the uniqueness of solutions for the boundary value problems we are presently considering.
More precisely, the following result establishes that any solution in 
$\dot{\mathscr{C}}^{\omega}(\mathbb{R}^n_{+},\mathbb{C}^M)$
can be obtained as a convolution of its trace with the associated Poisson kernel.

\begin{proposition}\label{fatou}
Let $L$ be a constant complex coefficient system as in \eqref{ars} satisfying the strong ellipticity 
condition stated in \eqref{LegHad}, and let $\omega$ be a growth function satisfying \eqref{omega-cond:b}. 
If $u\in\mathscr{C}^{\infty}(\mathbb{R}^n_{+},\mathbb{C}^M)\cap
\dot{\mathscr{C}}^{\omega}(\mathbb{R}^n_{+},\mathbb{C}^M)$ is a function satisfying $Lu=0$ in $\mathbb{R}^n_{+}$, 
then $\restr{u}{\partial\mathbb{R}^n_{+}}\in\dot{\mathscr{C}}^{\omega}(\mathbb{R}^{n-1},\mathbb{C}^M)$ and
\begin{equation}\label{claimeq}
u(x',t)=\left(P_t^L\ast(\restr{u}{\partial\mathbb{R}^n_{+}})\right)(x'),\qquad\forall\,(x',t)\in\mathbb{R}^n_{+}.
\end{equation}
\end{proposition}

\begin{proof}
Let $u\in\mathscr{C}^{\infty}(\mathbb{R}^n_{+},\mathbb{C}^M)\cap
\dot{\mathscr{C}}^{\omega}(\mathbb{R}^n_{+},\mathbb{C}^M)$ satisfy $Lu=0$ in $\mathbb{R}^n_{+}$. 
By Lemma~\ref{wlemma}\eqref{holderclosure}, it follows that $u$ can be continuously extended to 
a function (which we call again $u$) $u\in\dot{\mathscr{C}}^{\omega}(\overline{\mathbb{R}^n_{+}},\mathbb{C}^M)$. 
In particular, the trace $\restr{u}{\partial\mathbb{R}^n_{+}}$ is well-defined and belongs to the space  
$\dot{\mathscr{C}}^{\omega}(\mathbb{R}^{n-1},\mathbb{C}^M)$.
To proceed, fix an arbitrary $\varepsilon>0$ and define $u_\varepsilon=u(\cdot+\varepsilon e_n)$ 
in $\mathbb{R}^n_{+}$, where $e_n=(0,\dots,0,1)\in\mathbb{R}^n$. Then, by design, 
$u_\varepsilon\in\mathscr{C}^{\infty}(\overline{\mathbb{R}^n_{+}},\mathbb{C}^M)\cap
\dot{\mathscr{C}}^{\omega}(\overline{\mathbb{R}^n_{+}},\mathbb{C}^M)$, $Lu_{\varepsilon}=0$ in $\mathbb{R}^n_{+}$, 
and $[u_\varepsilon]_{\dot{\mathscr{C}^{\omega}}(\overline{\mathbb{R}^{n}_{+}},\mathbb{C}^M)}
\leq [u]_{\dot{\mathscr{C}^{\omega}}(\mathbb{R}^{n}_{+},\mathbb{C}^M)}$. Moreover, using 
Proposition~\ref{prop:proper-sols}\eqref{sup-Comega} and \eqref{wnonincreasing} we obtain
\begin{align}\label{ayttrf}
\sup_{(x',t)\in\mathbb{R}^n_{+}}|(\nabla u_\varepsilon)(x',t)| 
&=\sup_{(x',t)\in\mathbb{R}^n_{+}}|(\nabla u)(x',t+\varepsilon)| 
\nonumber\\[4pt]
&\leq C[u]_{\dot{\mathscr{C}^{\omega}}(\mathbb{R}^{n}_{+},\mathbb{C}^M)}
\sup_{(x',t)\in\mathbb{R}^n_{+}} \frac{\omega(t+\varepsilon)}{t+\varepsilon} 
\nonumber\\[4pt]
&\leq C\,C_\omega [u]_{\dot{\mathscr{C}^{\omega}}(\mathbb{R}^{n}_{+},\mathbb{C}^M)}
\frac{\omega(\varepsilon)}{\varepsilon}.
\end{align} 
This implies that $\nabla u_\varepsilon$ is bounded in $\mathbb{R}^n_{+}$, hence 
$u_\varepsilon\in W^{1,2}_{\rm bdd}(\mathbb{R}^n_{+},\mathbb{C}^M)$. 

Define next $f_\varepsilon(x'):=u(x',\varepsilon)\in\dot{\mathscr{C}}^{\omega}(\mathbb{R}^{n-1},\mathbb{C}^M)$ 
and $w_\varepsilon(x',t):=(P_t^L\ast f_\varepsilon)(x')$ for each $(x',t)\in\mathbb{R}^n_{+}$. 
Then, Proposition~\ref{propfholder} implies that that 
$w_\varepsilon\in\mathscr{C}^{\infty}(\mathbb{R}^n_{+},\mathbb{C}^M)\cap
\dot{\mathscr{C}}^{\omega}(\overline{\mathbb{R}^n_{+}},\mathbb{C}^M)$, $L w_\varepsilon=0$ 
in $\mathbb{R}^n_{+}$ and $\restr{w_\varepsilon}{\partial\mathbb{R}^n_{+}}= f_\varepsilon$. 
Moreover, for every pair of points $x',y'\in\mathbb{R}^{n-1}$ we have, on the one hand, 
\begin{equation}\label{ayfff}
|f_\varepsilon(x')-f_\varepsilon(y')|=|u(x',\varepsilon)-u(y',\varepsilon)|
\leq [u]_{\dot{\mathscr{C}^{\omega}}(\mathbb{R}^{n}_{+},\mathbb{C}^M)}\,\omega(|x'-y'|),
\end{equation}
and, on the other hand, using the Mean Value Theorem and  Proposition~\ref{prop:proper-sols}\eqref{sup-Comega},
\begin{align}\label{6ggff}
|f_\varepsilon (x')-f_\varepsilon(y')| &=|u(x',\varepsilon)-u(y',\varepsilon)| 
\nonumber\\[4pt]
&\leq|x'-y'|\sup_{z'\in[x',y']}|(\nabla u)(z',\varepsilon)|
\nonumber\\[4pt]
&\leq C\,|x'-y'|\,[u]_{\dot{\mathscr{C}^{\omega}}(\mathbb{R}^{n}_{+},\mathbb{C}^M)}
\frac{\omega(\varepsilon)}{\varepsilon}.
\end{align}  
Therefore, we conclude that $f_\varepsilon\in\dot{\mathscr{C}}^{\Psi}(\mathbb{R}^{n-1},\mathbb{C}^M)$, 
with norm depending (unfavorably) on the parameter $\varepsilon$, where the growth function $\Psi$ is given by 
\begin{equation}\label{psi}
\Psi(t):=\min\left\lbrace t,\frac{\omega(t)}{\omega(1)}\right\rbrace=\left\lbrace
\begin{array}{ll}
t &\text{ if }t\leq 1,
\\[4pt]
\omega(t)/\omega(1) &\text{ if }t>1.
\end{array}
\right.
\end{equation}
For every $R>1$ and $x=(x',t)$, let us now invoke \eqref{ConvDer}, 
\eqref{estimatea} and \eqref{oscmorrey}, with $\Psi$ in place of $\omega$, to write 
\begin{align}\label{q34r34fr}
&\int_{B(0,R)\cap\mathbb{R}^n_{+}}|(\nabla w_\varepsilon)(x)|^2\,dx
\leq\int_{B(0,R)\cap\mathbb{R}^n_{+}}\left(\frac{C}{t}\int_1^{\infty}\osc_1(f_\varepsilon;st)
\frac{ds}{s^2}\right)^2\,dx
\nonumber\\[4pt] 
&\qquad\leq C\norm{f_\varepsilon}_{\mathscr{E}^{\Psi,p}(\mathbb{R}^{n-1},\mathbb{C}^M)}
\int_{B(0,R)\cap\mathbb{R}^n_{+}}\left(\int_1^{\infty}\frac{\Psi(st)}{st}\frac{ds}{s}\right)^2\,dx
\nonumber\\[4pt]
&\qquad\leq C\norm{f_\varepsilon}_{\mathscr{E}^{\Psi,p}(\mathbb{R}^{n-1},\mathbb{C}^M)}
R^{n-1}\int_0^R\left(\int_t^\infty\frac{\Psi(s)}{s}\frac{ds}{s}\right)^2\,dt,
\end{align} 
and then use \eqref{omega-cond:b} to observe that
\begin{align}\label{ayredf5t}
\int_0^R\left(\int_t^\infty\frac{\Psi(s)}{s}\frac{ds}{s}\right)^2\,dt
&\leq\int_0^1\left(\int_t^1\frac{ds}{s}+\frac1{\omega(1)}\int_1^\infty\frac{\omega(s)}{s}\frac{ds}{s}\right)^2\,dt
\nonumber\\[4pt] 
&\qquad\quad+\int_1^R\left(\frac1{\omega(1)}\int_1^\infty\frac{\omega(s)}{s}\frac{ds}{s}\right)^2\,dt
\nonumber\\[4pt] 
&\leq\int_0^1\big(\log(1/t)+C_\omega\big)^2\,dt+(R-1)(C_\omega)^2<\infty,
\end{align}
Collectively, \eqref{q34r34fr} and \eqref{ayredf5t} show that 
$w_\varepsilon\in W^{1,2}_{\rm bdd}(\mathbb{R}^n_{+},\mathbb{C}^M)$.

We now consider $v_\varepsilon:=u_\varepsilon-w_\varepsilon\in\mathscr{C}^{\infty}(\mathbb{R}^n_{+},\mathbb{C}^M)
\cap\dot{\mathscr{C}}^{\omega}(\overline{\mathbb{R}^n_{+}},\mathbb{C}^M)\cap  
W^{1,2}_{\rm bdd}(\mathbb{R}^n_{+},\mathbb{C}^M)$, which satisfies $Lv_\varepsilon=0$ in $\mathbb{R}^n_{+}$ and  
$\restr{v_\varepsilon}{\partial\mathbb{R}^n_{+}}=0$. Hence, $\Tr\,v_\varepsilon=0$ on 
$\mathbb{R}^{n-1}$ (see \eqref{defi-trace}) and for each $x\in{\mathbb{R}}^n$ we have
\begin{align}\label{uatf}
|v_\varepsilon(x)| &\leq|v_\varepsilon(x)-v_\varepsilon(0)|+|v_\varepsilon(0)| 
\nonumber\\[4pt]
&\leq\max\big\{[v_\varepsilon]_{\dot{\mathscr{C}^{\omega}}(\mathbb{R}^{n}_{+},\mathbb{C}^M)}\,,\,
|v_\varepsilon(0)|\big\}(1+\omega(|x|)).
\end{align}
From this and Proposition~\ref{uniqprop} we then conclude that 
\begin{equation}\label{uniqeq2}
\sup_{\mathbb{R}^n_{+}\cap B(0,r)}|\nabla v_\varepsilon| 
\leq\frac{C}{r}\sup_{\mathbb{R}^n_{+}\cap B(0,2r)}|v_\varepsilon|\leq C_\varepsilon\frac{1+\omega(2r)}{r},
\end{equation} 
and from Lemma~\ref{wlemma}\eqref{wlimit} we see that the right side of \eqref{uniqeq2} tends to 
$0$ as $r\to\infty$. This forces $\nabla v_\varepsilon\equiv 0$, and since 
$v_\varepsilon\in\mathscr{C}^{\infty}(\mathbb{R}^n_{+},\mathbb{C}^M)\cap
\dot{\mathscr{C}}^{\omega}(\overline{\mathbb{R}^n_{+}},\mathbb{C}^M)$ 
with $\restr{v_\varepsilon}{\partial\mathbb{R}^n_{+}}=0$ we ultimately conclude that 
$v_\varepsilon\equiv 0$. Consequently,  
\begin{equation}\label{claimepsilon} 
u(x',t+\varepsilon)=(P_t^L\ast f_\varepsilon)(x'),\qquad\forall\,(x',t)\in\mathbb{R}^n_{+}. 
\end{equation}
Since, as noted earlier, 
$\restr{u}{\partial\mathbb{R}^n_{+}}\in\dot{\mathscr{C}}^{\omega}(\mathbb{R}^{n-1},\mathbb{C}^M)$,
for every $x'\in\mathbb{R}^{n-1}$ and $\varepsilon>0$ we may now write 
\begin{align}\label{ay6RR}
\big|u(x',t+\varepsilon)-\big(P_t^L\ast(\restr{u}{\partial\mathbb{R}^n_{+}})\big)(x')\big|
&=\big|\big(P_t^L * (f_\varepsilon-\restr{u}{\partial\mathbb{R}^n_{+}})\big)(x')\big|
\nonumber\\[4pt]
&\leq\|P_t^L\|_{L^1(\mathbb{R}^{n-1})}\sup_{y'\in\mathbb{R}^{n-1}}
\big|f_\varepsilon(y') -\restr{u}{\partial\mathbb{R}^n_{+}}(y')\big|
\nonumber\\[4pt]
&=\|P^L\|_{L^1(\mathbb{R}^{n-1})}\sup_{y'\in\mathbb{R}^{n-1}}
\big|u(y',\varepsilon) -\restr{u}{\partial\mathbb{R}^n_{+}}(y')\big|
\nonumber\\[4pt]
&\leq\|P^L\|_{L^1(\mathbb{R}^{n-1})}
[u]_{\dot{\mathscr{C}^{\omega}}(\overline{\mathbb{R}^{n}_{+}},\mathbb{C}^M)}\,\omega(\varepsilon).
\end{align}
From \eqref{PoissonDecay} we know that $\|P^L\|_{L^1(\mathbb{R}^{n-1})}<\infty$. 
Upon letting $\varepsilon\to 0^{+}$ and using that $\omega$ vanishes in the limit at the origin,
we see that \eqref{ay6RR} implies \eqref{claimeq}. This finishes the proof of Proposition~\ref{fatou}.
\end{proof}

\section{Well-Posedness Results}\label{section:well-general}

We are now ready to prove well-posedness results. We first consider the case in which 
the boundary data belong to generalized H\"older spaces and we note that, in such a scenario, 
the only requirement on the growth function is \eqref{omega-cond:b}.

\begin{theorem}\label{mainthmBMO2b}
Let $L$ be a constant complex coefficient $M\times M$ system as in \eqref{ars} satisfying the 
strong ellipticity condition \eqref{LegHad}. Also, let $\omega$ be a growth function satisfying 
\eqref{omega-cond:b}. Then the generalized H\"older Dirichlet problem for $L$ in 
$\mathbb{R}^{n}_{+}$, formulated as
\begin{equation}\label{BVP2b}
\left\lbrace
\begin{array}{l}
u\in\mathscr{C}^{\infty}(\mathbb{R}^{n}_{+},\mathbb{C}^M),
\\[4pt]
Lu=0\,\,\text{ in }\,\,\mathbb{R}^n_{+},
\\[4pt] 
\left[u\right]_{\dot{\mathscr{C}}^{\omega}(\mathbb{R}^n_{+},\mathbb{C}^M)}<\infty,
\\[6pt]
\restr{u}{\partial\mathbb{R}^n_{+}}^{{}^{\rm lim}}=f\in\dot{\mathscr{C}}^{\omega}(\mathbb{R}^{n-1},\mathbb{C}^M)
\,\,\text{ on }\,\,\mathbb{R}^{n-1},
\end{array}
\right.
\end{equation}
is well-posed. More specifically, there exists a unique solution which is given by  
\begin{equation}\label{eqn-Dir-Holder:u}
u(x',t)=(P_t^L\ast f)(x'),\qquad\forall\,(x',t)\in{\mathbb{R}}^n_{+},
\end{equation}
where $P^L$ denotes the Poisson kernel for the system $L$ in $\mathbb{R}^{n}_+$ from Theorem~\ref{PoissonConvolution}. 
In addition, $u$ extends to a function in $\dot{\mathscr{C}}^{\omega}(\overline{\mathbb{R}^n_{+}},\mathbb{C}^M)$ with 
$\restr{u}{\partial\mathbb{R}^n_{+}}=f$, and there exists a finite constant $C=C(n,L,\omega)\geq 1$ 
such that
\begin{equation}\label{mainthmBMOeq12b}
C^{-1}[f]_{\dot{\mathscr{C}}^{\omega}(\mathbb{R}^{n-1},\mathbb{C}^M)}
\leq[u]_{\dot{\mathscr{C}}^{\omega}(\mathbb{R}^n_{+},\mathbb{C}^M)}
\leq C[f]_{\dot{\mathscr{C}}^{\omega}(\mathbb{R}^{n-1},\mathbb{C}^M)}.
\end{equation} 
\end{theorem}

\begin{proof}
Given $f\in\dot{\mathscr{C}}^{\omega}(\mathbb{R}^{n-1},\mathbb{C}^M)$, define $u$ 
as in \eqref{eqn-Dir-Holder:u}. Proposition~\ref{propfholder} then implies that 
$u$ satisfies all conditions in \eqref{BVP2b}. Also, $u$ extends to a function 
in $\dot{\mathscr{C}}^{\omega}(\overline{\mathbb{R}^n_{+}},\mathbb{C}^M)$ with 
$\restr{u}{\partial\mathbb{R}^n_{+}}=f$, and the second inequality in 
\eqref{mainthmBMOeq12b} holds. Moreover, \eqref{holderclosureeq} yields
\begin{equation}\label{mainthmBMOeq3}
[f]_{\dot{\mathscr{C}}^{\omega}(\mathbb{R}^{n-1},\mathbb{C}^M)} 
=[\restr{u}{\partial\mathbb{R}^n_{+}}]_{\dot{\mathscr{C}}^{\omega}(\mathbb{R}^{n-1},\mathbb{C}^M)}
\leq[u]_{\dot{\mathscr{C}}^{\omega}(\overline{\mathbb{R}^{n}_{+}},\mathbb{C}^M)} 
\leq 2C_\omega [u]_{\dot{\mathscr{C}}^{\omega}(\mathbb{R}^{n}_{+},\mathbb{C}^M)},
\end{equation}
so that the first inequality in \eqref{mainthmBMOeq12b} follows.

It remains to prove that the solution is unique. However, this follows at once from 
Proposition~\ref{fatou}. Indeed, the first three conditions in \eqref{BVP2b} imply \eqref{claimeq} 
and since $\restr{u}{\partial\mathbb{R}^n_{+}}=f$ we conclude that necessarily  
$u(x',t)=\left(P_t^L\ast f\right)(x')$ for every $(x',t)\in\mathbb{R}^n_{+}$.
\end{proof}

Here is the well-posedness for the generalized Morrey-Campanato Dirichlet problem. 
In this case, the growth function is assumed to satisfy both \eqref{omega-cond:a} 
and \eqref{omega-cond:b}.

\begin{theorem}\label{mainthmBMO2}
Let $L$ be a constant complex coefficient $M\times M$ system as in \eqref{ars} satisfying the 
strong ellipticity condition \eqref{LegHad}. Fix $p\in[1,\infty)$ along with $q\in(0,\infty]$, 
and let $\omega$ be a growth function satisfying \eqref{omega-cond:a} and \eqref{omega-cond:b}. 
Then the generalized Morrey-Campanato Dirichlet problem for $L$ in $\mathbb{R}^{n}_{+}$, namely
\begin{equation}\label{BVP2}
\left\lbrace
\begin{array}{l}
u\in\mathscr{C}^{\infty}(\mathbb{R}^{n}_{+},\mathbb{C}^M),
\\[4pt]
Lu=0\,\,\text{ in }\,\,\mathbb{R}^n_{+},
\\[4pt] 
\norm{u}_{**}^{(\omega,q)}<\infty,
\\[4pt]
\restr{u}{\partial\mathbb{R}^n_{+}}^{{}^{\rm nt.lim}}=f\in\mathscr{E}^{\omega,p}(\mathbb{R}^{n-1},\mathbb{C}^M)
\,\,\text{ a.e. on }\,\,\mathbb{R}^{n-1},
\end{array}
\right.
\end{equation}
is well-posed. More specifically, there exists a unique solution which is given by  
\begin{equation}\label{eqn-Dir:M-C:u}
u(x',t)=(P_t^L\ast f)(x'),\qquad\forall\,(x',t)\in{\mathbb{R}}^n_{+},
\end{equation}
where $P^L$ denotes the Poisson kernel for the system $L$ in $\mathbb{R}^{n}_+$ from Theorem~\ref{PoissonConvolution}. 
Moreover, with $W$ defined as in \eqref{omega-cond:WDef}, the solution $u$ extends to a function in 
$\dot{\mathscr{C}}^{W}(\overline{\mathbb{R}^{n}_{+}},\mathbb{C}^M)$ with 
$\restr{u}{\partial\mathbb{R}^n_{+}}=f$ a.e. on $\mathbb{R}^{n-1}$, and there 
exists a finite constant $C=C(n,L,\omega,p,q)\geq 1$ for which 
\begin{equation}\label{mainthmBMOeq12}
C^{-1}\norm{f}_{\mathscr{E}^{W,p}(\mathbb{R}^{n-1},\mathbb{C}^M)}
\leq\norm{u}_{**}^{(\omega,q)}\leq C\norm{f}_{\mathscr{E}^{\omega,p}(\mathbb{R}^{n-1},\mathbb{C}^M)}.
\end{equation}
Furthermore, all results remain valid if $\norm{\cdot}_{**}^{(\omega,q)}$ 
is replaced everywhere by $\norm{\cdot}_{**}^{(\omega,\exp)}$.
\end{theorem}

\begin{proof}
Having fixed $f\in\mathscr{E}^{\omega,p}(\mathbb{R}^{n-1},\mathbb{C}^M)$, if $u$ is defined as in 
\eqref{eqn-Dir:M-C:u} then Proposition~\ref{propstep1} implies the validity of all conditions 
in \eqref{BVP2} and also of the second inequality in \eqref{mainthmBMOeq12} 
(even replacing $q$ by $\exp$). In the case $q=\infty$ we invoke 
Proposition~\ref{prop:proper-sols}\eqref{CW-car-q} to obtain that 
$u\in\dot{\mathscr{C}}^{W}(\overline{\mathbb{R}^{n}_{+}},\mathbb{C}^M)$ 
in the sense of Lemma~\ref{wlemma}\eqref{holderclosure}. Note that we also have
\begin{align}\label{8fafrafr}
[\restr{u}{\partial\mathbb{R}^n_{+}}]_{\dot{\mathscr{C}}^{W}(\mathbb{R}^{n-1},\mathbb{C}^M)} 
\leq[u]_{\dot{\mathscr{C}}^{W}(\overline{\mathbb{R}^{n}_{+}},\mathbb{C}^M)} 
\leq 2C_{W}[u]_{\dot{\mathscr{C}}^{W}(\mathbb{R}^{n}_{+},\mathbb{C}^M)}
\leq C(C_\omega)^4\norm{u}_{**}^{(\omega,\infty)}
\end{align}
thanks to \eqref{holderclosureeq} (for the growth function $W$), Lemma~\ref{wlemma3}, and \eqref{mainthmBMOeq4}. 

Given that, on the one hand,  
$\restr{u}{\partial\mathbb{R}^n_{+}}=\restr{u}{\partial\mathbb{R}^n_{+}}^{{}^{\rm nt.lim}}$ 
everywhere in $\mathbb{R}^{n-1}$ due to the fact that $u\in\dot{\mathscr{C}}^{W}(\overline{\mathbb{R}^{n}_{+}},\mathbb{C}^M)$,  
and that, on the other hand, $\restr{u}{\partial\mathbb{R}^n_{+}}^{{}^{\rm nt.lim}}=f $ a.e. in $\mathbb{R}^{n-1}$,
we conclude that $\restr{u}{\partial\mathbb{R}^n_{+}}=f$ a.e. in $\mathbb{R}^{n-1}$. In addition, 
\eqref{easyinclusion} (applied to $W$), Lemma~\ref{wlemma3}, and \eqref{8fafrafr} permit us to estimate  
\begin{multline}\label{mainthmBMOeq3b}
\norm{f}_{\mathscr{E}^{W,p}(\mathbb{R}^{n-1},\mathbb{C}^M)}
=\|\restr{u}{\partial\mathbb{R}^n_{+}}\|_{\mathscr{E}^{W,p}(\mathbb{R}^{n-1},\mathbb{C}^M)}
\leq\sqrt{n-1}\,C_W[\restr{u}{\partial\mathbb{R}^n_{+}}]_{\dot{\mathscr{C}^{W}}(\mathbb{R}^{n-1},\mathbb{C}^M)}
\\[4pt]
\leq C(C_\omega)^6\norm{u}_{**}^{(\omega,\infty)}\leq C(C_\omega)^7\norm{u}_{**}^{(\omega,q)}
\leq C(C_\omega)^7\norm{u}_{**}^{(\omega,\exp)},
\end{multline}
where $0<q<\infty$ and where we have also used Proposition~\ref{prop:proper-sols}\eqref{sup-q-car} and \eqref{uqeq:1}.

To prove that the solution is unique, we note that having $\norm{u}_{**}^{(\omega,q)}<\infty$ 
for a given $q\in(0,\infty]$, or even $\norm{u}_{**}^{(\omega,\exp)}<\infty$, implies that 
$\norm{u}_{**}^{(\omega,\infty)}<\infty$ by Proposition~\ref{prop:proper-sols}\eqref{sup-q-car} 
and \eqref{uqeq:1}. Having established this, Proposition~\ref{prop:proper-sols}\eqref{CW-car-q} applies 
and yields that $u\in\dot{\mathscr{C}}^{W}(\overline{\mathbb{R}^{n}_{+}},\mathbb{C}^M)$. 
Consequently, $\restr{u}{\partial\mathbb{R}^n_{+}}=\restr{u}{\partial\mathbb{R}^n_{+}}^{{}^{\rm nt.lim}}$ 
everywhere in $\mathbb{R}^{n-1}$, and if we also take into account the boundary condition from \eqref{BVP2}, 
we conclude that $\restr{u}{\partial\mathbb{R}^n_{+}}=f$ a.e. on $\mathbb{R}^{n-1}$. 
Moreover, since Lemma~\ref{wlemma3} ensures that $W$ is a growth function satisfying \eqref{omega-cond:b},
we may invoke Proposition~\ref{fatou} to write 
\begin{equation}\label{7ttrR}
u(x',t)=\left(P_t^L\ast(\restr{u}{\partial\mathbb{R}^n_{+}})\right)(x')=\left(P_t^L\ast f\right)(x'),
\qquad\forall\,(x',t)\in\mathbb{R}^n_{+}.
\end{equation}
The proof of the theorem is therefore finished. 
\end{proof}

\begin{remark}
Theorems~\ref{mainthmBMO2b} and \ref{mainthmBMO2} are closely related. To elaborate in this, 
fix a growth function $\omega$ satisfying \eqref{omega-cond:a} and \eqref{omega-cond:b}. 
From \eqref{easyinclusion} and Proposition~\ref{propstep1} it follows that, given any
$f\in\dot{\mathscr{C}}^{\omega}(\mathbb{R}^{n-1},\mathbb{C}^M)$, the unique solution of the 
boundary value problem \eqref{BVP2b}, i.e., $u(x',t)=(P_t^L\ast f)(x')$ for $(x',t)\in\mathbb{R}^{n}_{+}$, 
also solves \eqref{BVP2}, regarding now $f$ as a function in 
$\mathscr{E}^{\omega,p}(\mathbb{R}^{n-1},\mathbb{C}^M)$ (cf. \eqref{inclusionE})
with $p\in[1,\infty)$ and $q\in(0,\infty]$ arbitrary (and even with  
$\norm{\cdot}_{**}^{(\omega,q)}$ replaced by $\norm{\cdot}_{**}^{(\omega,\exp)}$). 
As such, $u$ satisfies \eqref{mainthmBMOeq12} whenever \eqref{omega-cond:a} holds. 

This being said, the fact that $f\in\mathscr{E}^{\omega,p}(\mathbb{R}^{n-1},\mathbb{C}^M)$ 
does not guarantee, in general, that the corresponding solution satisfies 
$u\in\dot{\mathscr{C}}^{\omega}(\mathbb{R}^n_{+},\mathbb{C}^M)$, even though we have 
established above that the solution to the boundary value problem \eqref{BVP2} belongs 
to $\dot{\mathscr{C}}^{W}(\mathbb{R}^n_{+},\mathbb{C}^M)$. Note that, as seen from 
\eqref{7rsSS}-\eqref{7tEEE} and \eqref{wW}, the space $\dot{\mathscr{C}}^{W}(\mathbb{R}^n_{+},\mathbb{C}^M)$
contains $\dot{\mathscr{C}}^{\omega}(\mathbb{R}^n_{+},\mathbb{C}^M)$.

This aspect is fully clarified with the help of Example~\ref{example} discussed further below, where 
we construct some growth function $\omega$ satisfying \eqref{omega-cond:a}, \eqref{omega-cond:b}, and 
for which the space $\mathscr{E}^{\omega,1}(\mathbb{R}^{n-1},\mathbb{C})$ 
is strictly bigger than $\dot{\mathscr{C}}^{\omega}(\mathbb{R}^{n-1},\mathbb{C})$. 
Its relevance for the issue at hand is as follows. Consider the boundary problem \eqref{BVP2} 
formulated with $L$ being the Laplacian in ${\mathbb{R}}^n$ and with 
$f\in\mathscr{E}^{\omega,1}(\mathbb{R}^{n-1},\mathbb{C})\setminus
\dot{\mathscr{C}}^{\omega}(\mathbb{R}^{n-1},\mathbb{C})$ as boundary datum. Its solution 
$u$ then necessarily satisfies $u\notin\dot{\mathscr{C}}^{\omega}(\mathbb{R}^n_{+},\mathbb{C})$,
for otherwise Lemma~\ref{wlemma}\eqref{holderclosure} would imply 
$u\in\dot{\mathscr{C}}^{\omega}(\overline{\mathbb{R}^n_{+}},\mathbb{C}^M)$ and since 
$\restr{u}{\partial\mathbb{R}^n_{+}}=\restr{u}{\partial\mathbb{R}^n_{+}}^{{}^{\rm nt.lim}}=f$ 
a.e. on $\mathbb{R}^{n-1}$ and $f$ is continuous in $\mathbb{R}^{n-1}$ we would conclude that 
$f$ coincides everywhere with 
$\restr{u}{\partial\mathbb{R}^n_{+}}\in\dot{\mathscr{C}}^{\omega}(\mathbb{R}^{n-1},\mathbb{C})$, a contradiction.
\end{remark}

In spite of the previous remark, Theorem~\ref{mainthmBMO} states that the boundary problems 
\eqref{BVP2b} and \eqref{BVP2} are actually equivalent under the stronger assumption \eqref{omega-cond:main}
on the growth function. Here is the proof of Theorem~\ref{mainthmBMO}.

\vskip 0.08in
\begin{proof}[Proof of Theorem~\ref{mainthmBMO}]
We start with the observation that \eqref{omega-cond:main} and Lemma~\ref{wlemma3} 
yield $C_0^{-1}W(t)\leq\omega(t)\leq C_0 W(t)$ for each $t\in(0,\infty)$. Therefore, 
\begin{equation}\label{spaceidentity1} 
\dot{\mathscr{C}^{\omega}}(\mathbb{R}^{n}_{+},\mathbb{C}^M)
=\dot{\mathscr{C}^{W}}(\mathbb{R}^{n}_{+},\mathbb{C}^M),\quad
\dot{\mathscr{C}^{\omega}}(\overline{\mathbb{R}^{n}_{+}},\mathbb{C}^M) 
=\dot{\mathscr{C}^{W}}(\overline{\mathbb{R}^{n}_{+}},\mathbb{C}^M), 
\end{equation} 
and 
\begin{equation}\label{spaceidentity2}
\mathscr{E}^{\omega,p}(\mathbb{R}^{n-1},\mathbb{C}^M) 
=\mathscr{E}^{W,p}(\mathbb{R}^{n-1},\mathbb{C}^M), 
\end{equation} 
as vector spaces, with equivalent norms.

Having made these identifications, we now proceed to observe that \eqref{bvp-Hol-Dir:main} follows 
directly from Theorem~\ref{mainthmBMO2b}, while \eqref{bvp-MC-Dir:main} is implied by 
Theorem~\ref{mainthmBMO2} with the help of \eqref{spaceidentity1} and \eqref{spaceidentity2}. 
To deal with \eqref{equiv:main}, we first observe that the left-to-right inclusion follows from 
Lemma~\ref{wlemma2}\eqref{inclusionitem}, whereas \eqref{easyinclusion} provides the accompanying 
estimate for the norms. For the converse inclusion, fix 
$f\in\mathscr{E}^{\omega,p}(\mathbb{R}^{n-1},\mathbb{C}^M)$ and set 
$u(x',t):=(P_t^L\ast f)(x')$ for every $(x',t)\in\mathbb{R}^{n}_{+}$. 
Theorem~\ref{mainthmBMO2} and \eqref{spaceidentity1} then imply that 
$u\in\dot{\mathscr{C}^{\omega}}(\overline{\mathbb{R}^{n}_{+}},\mathbb{C}^M)$ 
with $\restr{u}{\partial\mathbb{R}^n_{+}}= f$ a.e. on $\mathbb{R}^{n-1}$. 
Introduce $\widetilde{f}:=\restr{u}{\partial\mathbb{R}^n_{+}}$ and note that 
$\widetilde{f}\in\dot{\mathscr{C}^{\omega}}(\mathbb{R}^{n-1},\mathbb{C}^M)$ 
with $\widetilde{f}=f$ a.e. on $\mathbb{R}^{n-1}$. 
Then $u(x',t)=(P_t^L\ast\widetilde{f})(x')$ and, thanks to \eqref{mainthmBMOeq12b}, 
\eqref{relation-all-norms:1}, and \eqref{mainthmBMOeq12}, we have 
\begin{equation}\label{y6tgg}
[\widetilde{f}\,]_{\dot{\mathscr{C}}^{\omega}(\mathbb{R}^{n-1},\mathbb{C}^M)} 
\leq C[u]_{\dot{\mathscr{C}}^{\omega}(\mathbb{R}^n_{+},\mathbb{C}^M)} 
\leq C\norm{f}_{\mathscr{E}^{\omega,p}(\mathbb{R}^{n-1},\mathbb{C}^M)}.
\end{equation}
This completes the treatment of \eqref{equiv:main}, and finishes the proof of Theorem~\ref{mainthmBMO}.
\end{proof}

We are now in a position to give the proof of Corollary~\ref{maincorBMO}.

\vskip 0.08in
\begin{proof}[Proof of Corollary~\ref{maincorBMO}]
We start by observing that \eqref{eq:qfrafr} is a direct consequence of  
Proposition~\ref{prop:proper-sols}\eqref{all-comp}. In particular, the last three equalities 
in \eqref{HLMO} follow at once. Also, the fact that the second set in the first 
line of \eqref{HLMO} is contained in $\dot{\mathscr{C}}^{\omega}(\mathbb{R}^{n-1},\mathbb{C}^M)$ 
is a consequence of Lemma~\ref{wlemma}\eqref{holderclosure}. 
Finally, given any $f\in\dot{\mathscr{C}}^{\omega}(\mathbb{R}^{n-1},\mathbb{C}^M)$, if $u$ is the solution 
of \eqref{BVPb} corresponding to this choice of boundary datum, then $\restr{u}{\partial\mathbb{R}^n_{+}}=f$ 
and $u$ also satisfies the required conditions to be an element in the second set displayed in \eqref{HLMO}.
\end{proof}

The following example shows that conditions \eqref{omega-cond:a} and \eqref{omega-cond:b} 
do not imply \eqref{maincorBMOeq}.

\begin{example}\label{example}	
Fix two real numbers $\alpha,\beta\in(0,1)$ and consider the growth function 
$\omega:(0,\infty)\to(0,\infty)$ defined for each $t>0$ as
\begin{equation}\label{6tgtfa}
\omega(t):=
\left\lbrace
\begin{array}{ll}
t^{\alpha}, &\text{ if }t\leq 1,
\\[4pt]
1+(\log t)^{\beta}, &\text{ if }t>1.
\end{array}
\right.
\end{equation} 
Clearly, $\omega$ satisfies \eqref{omega-cond:a}, and we also claim that 
$\omega$ satisfies \eqref{omega-cond:b}. Indeed, for $t\leq 1$, 
\begin{equation}\label{65tr}
\int_t^{\infty}\frac{\omega(s)}{s}\frac{ds}{s}=\int_{t}^1 s^{\alpha-1}\frac{ds}{s} 
+\int_1^{\infty}\frac{1+(\log s)^{\beta}}{s^2}\,ds\leq C(t^{\alpha-1}+1)\leq 2C\,t^{\alpha-1}.
\end{equation}
For $t\in[1,\infty)$, define
\begin{equation}\label{65fff}
F(t):=\frac{\displaystyle t\int_t^{\infty}\frac{\omega(s)}{s}\frac{ds}{s}}{\omega(t)} 
=\frac{\displaystyle\int_t^{\infty}\frac{1+(\log s)^{\beta}}{s^2\,ds}}{\displaystyle\frac{1+(\log t)^{\beta}}{t}},
\end{equation}
which is a continuous function in $[1,\infty)$ and satisfies $F(1)<\infty$. 
Moreover, using L'H\^{o}pital's rule,
\begin{equation}\label{6tfff}
\lim_{t\to\infty}F(t)=\lim_{t\to\infty}\frac{-(1+(\log t)^\beta)}{\beta(\log t)^{\beta-1}-(1+(\log t)^\beta)}=1.
\end{equation} 
Hence, $F$ is bounded, which amounts to having $\omega$ satisfy \eqref{omega-cond:b}. 
The function $W$, defined as in \eqref{omega-cond:WDef}, is currently given by
\begin{equation}\label{63dd}
W(t)=
\left\lbrace
\begin{array}{ll}
\dfrac1{\alpha} t^{\alpha}, &\text{ if }t\leq 1,
\\[10pt]
\dfrac{1}{\alpha}+\dfrac1{\beta+1}(\log t)^{\beta+1}+\log t, &\text{ if }t>1.
\end{array}
\right.
\end{equation}
Since \eqref{omega-cond:main} would imply $W(t)\leq C\omega(t)$ which is not the case for $t$ sufficiently large, 
we conclude that the growth function $\omega$ satisfies \eqref{omega-cond:a} and \eqref{omega-cond:b}
but it does not satisfy \eqref{omega-cond:main}.

For this choice of $\omega$, we now proceed to check that 
$\mathscr{E}^{\omega,1}(\mathbb{R}^{n-1},\mathbb{C})\neq\dot{\mathscr{C}}^{\omega}(\mathbb{R}^{n-1},\mathbb{C})$.
To this end, consider the function
\begin{equation}
f(x):=\log_{+}|x_1|,\qquad\forall\,x=(x_1,x_2,\dots,x_{n-1})\in\mathbb{R}^{n-1},
\end{equation}
where $\log_{+}t:=\max\{0,\log t\}$. With $e_1=:(1,0,\dots,0)\in\mathbb{R}^{n-1}$ we then have
\begin{equation}\label{6r4d5}
\sup_{x\neq y}\frac{|f(x)-f(y)|}{\omega(|x-y|)} 
\geq\lim_{x_1\to\infty}\frac{|f(x_1e_1)-f(e_1)|}{\omega(|x_1e_1-e_1|)} 
=\lim_{x_1\to\infty}\frac{\log x_1}{1+(\log (x_1-1))^{\beta}}=\infty,
\end{equation} 
since $\beta<1$. This means that $f\notin\dot{\mathscr{C}}^{\omega}(\mathbb{R}^{n-1},\mathbb{C})$. 
To prove that $f\in\mathscr{E}^{\omega,1}(\mathbb{R}^{n-1},\mathbb{C})$, consider 
$\widetilde{Q}:=(a,b)\times Q \subset\mathbb{R}^{n-1}$, where $Q$ is an arbitrary cube in 
$\mathbb{R}^{n-2}$ and $a,b\in\mathbb{R}$ are arbitrary numbers satisfying $a<b$. Then,
\begin{align}\label{eq:example1}
\norm{f}_{\mathscr{E}^{\omega,1}(\mathbb{R}^{n-1},\mathbb{C})} 
&\leq\sup_{\widetilde{Q}\subset\mathbb{R}^{n-1}}\frac1{\omega(\ell(\widetilde{Q}))}
\fint_{\widetilde{Q}}\fint_{\widetilde{Q}}|f(x)-f(y)|\,dx\,dy 
\nonumber\\[4pt]
&\leq\sup_{a<b}\frac{1}{\omega(b-a)}H(a,b),
\end{align}
where 
\begin{equation}\label{eq:example1-hhh}
H(a,b):=\fint_a^b\fint_a^b\big|\log_{+}|x_1|-\log_{+}|y_1|\big|\,dx_1\,dy_1.
\end{equation}
We shall now prove that the right-hand side of \eqref{eq:example1} is finite considering several different cases. 

\vskip 0.08in
\noindent{\bf Case~I: $\boldsymbol{ 1\leq a<b}$.} In this scenario, define 
\begin{align}\label{eq:example2-ggg}
G(\lambda):=1+2\lambda-2\lambda(\lambda+1)\log\left(1+\frac1{\lambda}\right),\quad\forall\,\lambda>0.
\end{align} 
Note that $G$ is continuous in $(0,\infty$), $G(0)=1$, and by 
L'H\^{o}pital's rule, $\lim_{\lambda\to\infty}G(\lambda)=0$, hence $G$ is bounded.
Also, 
\begin{align}\label{eq:example2}
H(a,b)=\frac{b^2-a^2-2ab\log (b/a)}{(b-a)^2}=G\big(a/(b-a)\big).
\end{align} 
Consequently, whenever $b-a\geq 1$ we have
\begin{equation}\label{eq:example3}
H(a,b)=G\big(a/(b-a)\big)\leq C\leq C\big(1+(\log (b-a)\big)^{\beta})=C\omega(b-a).
\end{equation} 
Again by L'H\^{o}pital's rule, $\lim_{\lambda\to\infty}\lambda^{\alpha} G(\lambda)=0$, 
hence $\lambda^{\alpha}G(\lambda)\leq C$ for every $\lambda>0$. Therefore, whenever $0<b-a<1$ we may write 
\begin{equation}\label{eq:example4}
H(a,b)=G\big(a/(b-a)\big)\leq C\left(\frac{b-a}{a}\right)^{\alpha}\leq C(b-a)^\alpha=C\omega(b-a).
\end{equation}
All these show that $H(a,b)\leq C\omega(b-a)$ in this case.

\vskip 0.08in
\noindent{\bf Case~II: $\boldsymbol{a<b\leq -1}$.} This case is analogous to the previous one by symmetry.

\vskip 0.08in
\noindent{\bf Case~III: $\boldsymbol{-1\leq a<b\leq 1}$.} This case is straightforward 
since $H(a,b)=0$, given that $\log_{+}|x_1|=\log_{+}|y_1|=0$ whenever $a<x_1,y_1<b$. 

\vskip 0.08in
\noindent{\bf Case~IV: $\boldsymbol{-1<a<1<b}$.} In this case we obtain
\begin{align}\label{eq:example5}
H(a,b) &=\frac{1}{(b-a)^2}\int_1^b\int_1^b|\log x_1-\log y_1|\,dx_1\,dy_1
\nonumber\\[4pt] 
&\qquad+\frac1{(b-a)^2}\int_a^1\int_1^b\log x_1\,dx_1\,dy_1 
+\frac{1}{(b-a)^2}\int_1^b\int_a^1\log y_1\,dx_1\,dy_1
\nonumber\\[4pt] 
&\leq\frac{(b-1)^2}{(b-a)^2}H(1,b)+2\frac{(1-a)(b\log b-b+1)}{(b-a)^2}.
\end{align} 
For the first term in the right-hand side of \eqref{eq:example5}, we use \eqref{eq:example3} 
and \eqref{eq:example4} (written with $a:=1$) and obtain, keeping in mind that in this case $a<1$,
\begin{equation}\label{765FF}
\frac{(b-1)^2}{(b-a)^2}H(1,b)\leq C\omega(b-1)\left(\frac{b-1}{b-a}\right)^2\leq C\omega(b-a).
\end{equation}
To bound the second term in the right-hand side of \eqref{eq:example5}, we first use the 
fact that $1-a<2$ and $\log t\leq t-1$ for every $t\geq 1$ to obtain
\begin{align}
\frac{(1-a)(b\log b-b+1)}{(b-a)^2} &\leq 2\frac{b(b-1)-b+1}{(b-a)^2}
\leq 2\left(\frac{b-1}{b-a}\right)^2\leq 2 
\nonumber\\[4pt]
&\leq 2\big(1+(\log(b-a))^{\beta}\big)=2\omega(b-a),
\end{align}
whenever $b-a\geq 1$. To study the case when $b-a<1$, bring in the auxiliary function
\begin{equation}\label{ttFF}
\widetilde{G}(\lambda):=\frac{\lambda\log \lambda-\lambda+1}{(\lambda-1)^{1+\alpha}},\qquad\forall\lambda>1.
\end{equation}
By L'H\^{o}pital's rule, $\lim_{\lambda\to 1^{+}}\widetilde{G}(\lambda)=0$, hence $\widetilde{G}(\lambda)\leq C$ 
for each $\lambda\in(1,2]$. If $b-a<1$, we clearly have $1<b\leq 2$ which, in turn, permits us to estimate
\begin{align}\label{6ttffa}
\frac{(1-a)(b\log b-b+1)}{(b-a)^2}
&=\frac{(1-a)(b-1)^{1+\alpha}G(b)}{(b-a)^2}
\nonumber\\[4pt]
&\leq\frac{C(b-1)^{1+\alpha}}{b-a}\leq C(b-a)^{\alpha}=C\omega(b-a).
\end{align}
Consequently, we have obtained that $H(a,b)\leq C\omega(b-a)$ in this case as well. 

\vskip 0.08in
\noindent{\bf Case~V: $\boldsymbol{a<-1<b<1}$.} This is analogue to Case~IV, again by symmetry.

\vskip 0.08in
\noindent{\bf Case~VI: $\boldsymbol{a<-1,\,\,b>1}$.} We break the interval $(a,b)$ into two 
intervals $(a,0)$ and $(0,b)$ to obtain
\begin{equation}\label{6t5ff}
H(a,b)\leq\frac{1}{(b-a)^2}(I+II),
\end{equation}
where, using Case~IV and Case~V, 
\begin{align}\label{6g5f5}
I &:=\int_a^0\int_a^0\big|\log_{+}|x_1|-\log_{+}|y_1|\big|\,dx_1\,dy_1 
+\int_0^b\int_0^b|\log_{+}x_1-\log_{+}y_1|\, dx_1\,dy_1
\nonumber\\[4pt] 
&=H(a,0)(0-a)^2+H(0,b)(b-0)^2\leq C|a|^2\omega(|a|)+C\,b^2\omega(b)
\nonumber\\[4pt] 
&\leq 2C(b-a)^2\omega(b-a).
\end{align} 
Similarly, by Case~IV,
\begin{align}\label{9g9h9h}
II &:=\int_a^0\int_0^b\big|\log_{+}|x_1|-\log_{+}|y_1|\big|\, dx_1\,dy_1 
+\int_0^b\int_a^0\big|\log_{+}|x_1|-\log_{+}|y_1|\big|\,dx_1\,dy_1
\nonumber\\[4pt] 
&\leq 2\big(\max\{|a|,b\}-0\big)^2H\big(\max\{|a|,b\},0\big)
\nonumber\\[4pt] 
&\leq 2C\max\{|a|,b\}^2\,\omega(\max\{|a|,b\}) 
\nonumber\\[4pt] 
&\leq 2C(b-a)^2\,\omega(b-a).
\end{align}
Thus, $H(a,b)\leq C\omega(b-a)$ in this case also.

Collectively, the results in Cases~I-VI prove that 
$f\in\mathscr{E}^{\omega,1}(\mathbb{R}^{n-1},\mathbb{C})$.
\end{example}

\appendix

\section{John-Nirenberg's Inequality Adapted to Growth Functions}\label{section:JN}

In  what follows we assume that all cubes are half-open, that is, they can be written 
in the form $Q=[a_1,a_1+\ell(Q))\times\dots\times [a_{n-1},a_{n-1}+\ell(Q))$ with 
$a_i\in\mathbb{R}^{n-1}$ and $\ell(Q)>0$. Notice that since $\partial Q$ has Lebesgue 
measure zero the half-open assumption is harmless. Subdividing dydically yields  
the collection of (half-open) dyadic-subcubes of a given cube $Q$, which we shall denote by $\mathbb{D}_{Q}$.
For the following statement, and with the aim of considering global results, it is also 
convenient to consider the case $Q=\mathbb{R}^{n-1}$ in which scenario we take $\mathbb{D}_{Q}$ 
to be the classical dyadic grid generated by $[0,1)^{n-1}$, or any other dyadic grid. Let us also recall 
the definition of the dyadic Hardy-Littlewood maximal function localized to a given cube $Q$, i.e., 
\begin{equation}\label{eq:def-Md}
\big(M^d_Q f\big)(x):=\sup_{x\in Q'\in\mathbb{D}_Q}\fint_{Q'}|f(y')|\,dy',\qquad x\in Q,
\end{equation}
for each $f\in L^1(Q)$. The following result is an extension of the John-Nirenberg inequality 
obtained in \cite{Hofmann-Mayboroda, HofMarMay} (when $\omega\equiv 1$) adapted to our growth function.  

\begin{lemma}\label{lem:appendix}
Let $F\in L^2_{\loc}(\mathbb{R}^n_{+})$ and let $\varphi:[0,\infty)\to[0,\infty)$ be a non-decreasing function. 
Let $Q_0\subset\mathbb{R}^{n-1}$ be an arbitrary half-open cube, or $Q_0=\mathbb{R}^{n-1}$. Assume that there 
are numbers $\alpha\in(0,1)$ and $N\in(0,\infty)$ such that
\begin{equation}\label{eq:appendixA1}
\bigg|\bigg\{x'\in Q:\,\frac{1}{\varphi(\ell(Q))}\bigg(\int_0^{\ell(Q)}\int_{|x'-y'|<\kappa s}
|F(y',s)|^2\,dy'\frac{ds}{s^n}\bigg)^{1/2}>N\bigg\}\bigg|\leq\alpha|Q|
\end{equation} 
for every cube $Q\in\mathbb{D}_{Q_0}$ and $\kappa:= 1+2\sqrt{n-1}$. Then, for every $t>0$ 
\begin{multline}\label{eq:F-levelsets}
\sup_{Q\in\mathbb{D}_{Q_0}}\frac{1}{|Q|}\bigg|\bigg\{x'\in Q:\,\frac{1}{\varphi(\ell(Q))}
\bigg(\int_0^{\ell(Q)}\int_{|x'-y'|<\kappa s}|F(y',s)|^2\,dy'\frac{ds}{s^n}\bigg)^{1/2}>t\bigg\}\bigg|
\\[4pt]
\leq\frac{1}{\alpha}\,e^{-\tfrac{\log(\alpha^{-1})}{N}\,t}.
\end{multline}

Hence, for each $q\in(0,\infty)$ there exists a finite constant $C=C(\alpha,q)\geq 1$ such that
\begin{equation}\label{eq:appendixA2}
\sup_{Q\in\mathbb{D}_{Q_0}}\frac{1}{\varphi(\ell(Q))}\bigg(\fint_Q\bigg(\int_{0}^{\ell(Q)}
\int_{|x'-y'|<s}|F(y',s)|^2\,dy'\frac{ds}{s^n}\bigg)^{q/2}\,dx'\bigg)^{1/q}\leq CN.
\end{equation}
Moreover, there exists some finite $C=C(\alpha)\geq 1$ such that
\begin{equation}\label{eq:appendixA2-exp}
\sup_{Q\in\mathbb{D}_{Q_0}}\frac{1}{\varphi(\ell(Q))} 
\bigg\|\bigg(\int_{0}^{\ell(Q)}\int_{|\,\cdot\,-y'|<s}|F(y',s)|^2\,dy'\frac{ds}{s^n}\bigg)^{1/2}\bigg\|_{\exp L,Q}\leq CN.
\end{equation}
\end{lemma}

\medskip

The previous result can be proved using the arguments in \cite{Hofmann-Mayboroda, HofMarMay} 
with appropriate modifications. Here we present an alternative abstract argument 
based on ideas that go back to Calder\'on, as presented in \cite{Neri} 
(see also \cite{Perez-Paseky, Perez-Lerner}). This also contains as a particular 
case the classical John-Nirenberg result concerning the exponential integrability of {\rm BMO} functions. 

\begin{proposition}\label{prop:general-exp-SI}
Let $Q_0\subset\mathbb{R}^{n-1}$ be an arbitrary half-open cube, or $Q_0=\mathbb{R}^{n-1}$. 
For every $Q\in\mathbb{D}_{Q_0}$ assume that there exist two non-negative functions 
$G_Q,\,H_Q\in L^1_{\loc}(\mathbb{R}^{n-1})$ such that
\begin{equation}\label{eq:FQ-GQ:1}
G_Q(x')\leq H_Q(x')\,\,\text{ for almost every }\,\,x'\in Q,
\end{equation}
and, for every $Q'\in\mathbb{D}_{Q}\setminus\{Q\}$,
\begin{equation}\label{eq:FQ-GQ:2}
G_Q(x')\leq G_{Q'}(x')+H_Q(y')\,\,\text{ for a.e. $x\in Q'$ and for a.e. }\,\,y'\in\widehat{Q'},
\end{equation}
where $\widehat{Q'}$ is the dyadic parent of $Q'$. For each $\alpha\in(0,1)$ define
\begin{equation}\label{eq:defi-mfrak}
\mathfrak{m}_{\alpha}:=\sup_{Q\in\mathbb{D}_{Q_0}}\inf\big\{\lambda>0:\,|\{x'\in Q: H_Q(x')>\lambda\}|
\leq\alpha|Q|\big\}.
\end{equation}

Then, for every $\alpha\in(0,1)$ one has
\begin{equation}\label{eq:JN-point}
\sup_{Q\in\mathbb{D}_{Q_0}}\frac{|\{x\in Q:\,G_Q(x')>t\}|}{|Q|}
\leq\frac{1}{\alpha}\,e^{-\log(\alpha^{-1})\,\tfrac{t}{\mathfrak{m}_{\alpha}}},\qquad\forall\,t>0.
\end{equation}
As a consequence, 
\begin{equation}\label{eq:exp-L-SI}
\sup_{Q\in\mathbb{D}_{Q_0}}\|G_Q\|_{\exp L,Q}\leq\frac{1+\alpha^{-1}}{\log (\alpha^{-1})}\,\mathfrak{m}_{\alpha}
\end{equation}
and for every $q\in(0,\infty)$ there exists a finite constant $C=C(q)>0$ such that
\begin{equation}\label{eq:Lq-SI}
\sup_{Q\in\mathbb{D}_{Q_0}}\Big(\fint_{Q} G_Q(x')^q\,dx'\Big)^{1/q}
\leq C_q\frac{1}{\alpha^{1/q}\log (\alpha^{-1})}\,\mathfrak{m}_{\alpha}.
\end{equation}
\end{proposition}

\medskip

Before proving this result and Lemma~\ref{lem:appendix}, let us illustrate how 
Proposition~\ref{prop:general-exp-SI} yields the classical John-Nirenberg result regarding the 
exponential integrability of {\rm BMO} functions. Concretely, pick $f\in\BMO(\mathbb{R}^{n-1})$. 
Fix an arbitrary cube $Q_0$ and for every $Q\in\mathbb{D}_{Q_0}$ define $G_Q:=|f-f_Q|$ and 
$H_Q:=2^{n-1}M^d_Q\big(|f-f_Q|\big)$ (cf. \eqref{eq:def-Md}).
Clearly, \eqref{eq:FQ-GQ:1} holds by Lebesgue's Differentiation Theorem. 
Moreover, for every $Q'\in\mathbb{D}_{Q}\setminus Q$, $x'\in Q'$, and $y'\in\widehat{Q'}$ we have
\begin{multline}\label{5f3r55}
G_Q(x')\leq|f(x')-f_{Q'}|+|f_Q'-f_Q|\leq G_{Q'}(x')+\fint_{Q'}|f(z')-f_Q|\,dz'
\\[4pt]
\leq G_{Q'}(x')+2^{n-1}\fint_{\widehat{Q'}}|f(z')-f_Q|\,dz'
\leq G_{Q'}(x')+H_{Q'}(y'),
\end{multline}
and \eqref{eq:FQ-GQ:2} follows. Going further, by the weak-type $(1,1)$ of the 
dyadic Hardy-Littlewood maximal function, for every $\lambda>0$ we may write 
\begin{equation}\label{6h6g6}
|\{x'\in Q:\,H_Q(x')>\lambda\}|\leq\frac{2^{n-1}}{\lambda}\int_{Q}|f(y')-f_Q|\,dy'
\leq\frac{2^{n-1}\|f\|_{\BMO(\mathbb{R}^{n-1})}}{\lambda}|Q|.
\end{equation}
In particular, choosing for instance $\alpha:=e^{-1}$, if we use the previous 
estimate with $\lambda:=2^{n-1}\|f\|_{\BMO(\mathbb{R}^{n-1})}/\alpha$ we obtain 
$\mathfrak{m}_\alpha\le 2^{n-1}\|f\|_{\BMO(\mathbb{R}^{n-1})}/\alpha$. 
Thus, \eqref{eq:JN-point} yields 
\begin{align}\label{6g444}
\frac{|\{x\in Q_0:\,|f(x')-f_{Q_0}|>t\}|}{|Q_0|}
&\leq\frac1\alpha\,e^{-\tfrac{\alpha\log(\alpha^{-1})}{2^{n-1}}\tfrac{t}{\|f\|_{\BMO(\mathbb{R}^{n-1})}}}
\nonumber\\[6pt]
&=e\cdot e^{-\tfrac1 {2^{n-1}e}\tfrac{t}{\|f\|_{\BMO(\mathbb{R}^{n-1})}}}
\end{align}
while \eqref{eq:exp-L-SI} gives 
\begin{equation}\label{63222}
\|f-f_{Q_0}\|_{\exp L,Q_0}\leq(1+e)\,e\,2^{n-1}\|f\|_{\BMO(\mathbb{R}^{n-1})}
\end{equation}
which are the well-known John-Nirenberg inequalities. 

We now turn to the proof of Lemma~\ref{lem:appendix}. 

\vskip 0.08in
\begin{proof}[Proof of Lemma~\ref{lem:appendix}]
Let $F$, $\alpha$, and $N$ be fixed as in the statement of the lemma.
For every $Q\in\mathbb{D}_{Q_0}$ and $x' \in \mathbb{R}^{n-1}$,
define 
\begin{equation}\label{53222}
G_Q(x'):=\frac{1}{\varphi(\ell(Q))} \Big( \int_{0}^{\ell(Q)}
\int_{|x'-z'|<s}|F(z',s)|^2\,dz'\frac{ds}{s^n} \Big)^{1/2}
\end{equation}
and
\begin{equation}
H_Q(x')=\frac{1}{\varphi(\ell(Q))} \Big(
\int_{0}^{\ell(Q)}\int_{|x'-z'|<\kappa s}|F(z',s)|^2\,dz'\frac{ds}{s^n} \Big)^{1/2}. 
\end{equation}
Note that \eqref{eq:FQ-GQ:1} is trivially verified since $\kappa>1$. To proceed, fix  
$Q'\in\mathbb{D}_{Q}$ along with $x'\in Q'$ and $y'\in\widehat{Q'}$. If $|x'-z'|<s$ 
with $\ell(Q')\leq s\le\ell(Q)$ then 
\begin{equation}\label{g5FFF}
|y'-z'|\leq|y'-x'|+|x'-z'|<2\sqrt{n-1}\,\ell(Q')+s\leq\kappa s.
\end{equation}
Therefore, since $\varphi$ is non-decreasing,
\begin{multline}\label{522rFF}
G_Q(x')\leq\frac{\varphi(\ell(Q'))}{\varphi(\ell(Q))}G_{Q'}(x')
+\frac{1}{\varphi(\ell(Q))} 
\Big( \int_{\ell(Q')}^{\ell(Q)}\int_{|x'-z'|<s}|F(z',s)|^2\,dz'\frac{ds}{s^n} \Big)^{1/2}
\\[4pt]
\leq G_{Q'}(x')+H_Q(y'),
\end{multline}
establishing \eqref{eq:FQ-GQ:2}. Moreover, \eqref{eq:appendixA1} gives immediately that 
$\mathfrak{m}_\alpha\le N$. Granted this, \eqref{eq:JN-point}, \eqref{eq:Lq-SI}, and \eqref{eq:exp-L-SI}, 
(with $\alpha\in(0,1)$ given by \eqref{eq:appendixA1}) prove, 
respectively \eqref{eq:F-levelsets}, \eqref{eq:appendixA2}, and \eqref{eq:appendixA2-exp}. 
\end{proof}

Finally, we give the proof of Proposition~\ref{prop:general-exp-SI}.

\vskip 0.08in
\begin{proof}[Proof of Proposition~\ref{prop:general-exp-SI}]
We start by introducing some notation. Set
\begin{equation}
\Xi(t):=\sup_{Q\in\mathbb{D}_{Q_0}}\frac{|E_Q(t)|}{|Q|}
:=\sup_{Q\in\mathbb{D}_{Q_0}}\frac{|\{x'\in Q: G_Q(x')>t\}|}{|Q|},\qquad 0<t<\infty.
\end{equation}
Fix $\alpha\in(0,1)$, let $\varepsilon>0$ be arbitrary, and write 
$\lambda_\varepsilon=\mathfrak{m}_\alpha+\varepsilon$. From \eqref{eq:defi-mfrak} it follows that
\begin{equation}\label{eq:esti-unif-level-sets}
|F_{Q,\varepsilon}|:=|\{x'\in Q:\,H_Q(x')>\lambda_\varepsilon\}|\leq\alpha|Q|,
\qquad\forall\,Q\in\mathbb{D}_{Q_0}.
\end{equation}
Fix now $Q\in\mathbb{D}_{Q_0}$, $\beta\in(\alpha,1)$ 
(we will eventually let $\beta \to 1^{+})$ and
set
\begin{equation}\label{643rUHG}
\Omega_Q:=\{x'\in Q:\,M_Q^d(1_{F_{Q,\varepsilon}})(x')>\beta\}.
\end{equation}
Note that \eqref{eq:esti-unif-level-sets} ensures that 
\begin{equation}\label{6f4rdDD}
\fint_Q 1_{F_{Q,\varepsilon}}(y')\,dy'=\frac{|F_{Q,\varepsilon}|}{|Q|}\leq\alpha<\beta,
\end{equation}
hence we can extract a family of pairwise disjoint stopping-time cubes 
$\{Q_j\}_j\subset\mathbb{D}_{Q}\setminus\{Q\}$ so that $\Omega_Q=\cup_j Q_j$ and for every $j$
\begin{equation}\label{eq:stopping-time}
\frac{|F_{Q,\varepsilon}\cap Q_j|}{|Q_j|}>\beta,
\quad\frac{|F_{Q,\varepsilon}\cap Q'|}{|Q'|}\leq\beta,
\quad Q_j\subsetneq Q'\in\mathbb{D}_Q.
\end{equation}

Let $t>\lambda_\varepsilon$ and note that \eqref{eq:FQ-GQ:1} gives
\begin{equation}\label{64321-AA}
\lambda_\varepsilon<t<G_Q(x')\leq H_Q(x')\,\,\text{ for a.e. }\,\,x'\in E_{Q}(t).
\end{equation}
which implies that
\begin{equation}\label{tdr-aafDD}
\beta<1={\mathbf{1}}_{F_{Q,\varepsilon}}(x')\leq M_Q^d\big({\mathbf{1}}_{F_{Q,\varepsilon}}\big)(x')
\,\,\text{ for a.e. }\,\,x'\in E_{Q}(t).
\end{equation}
Hence,
\begin{equation}\label{6tfFDdad}
|E_Q(t)|=|E_Q(t)\cap\Omega_Q|=\sum_j |E_Q(t)\cap Q_j|.
\end{equation}
For every $j$, by the second estimate in \eqref{eq:stopping-time} 
applied to $\widehat{Q}_j$, the dyadic parent of $Q_j$, we have  
$|F_{Q,\varepsilon}\cap\widehat{Q}_j|/|\widehat{Q}_j|\leq\beta<1$, therefore 
$|\widehat{Q}_j\setminus F_{Q,\varepsilon}|/|\widehat{Q}_j|>1-\beta>0$. 
In particular, \eqref{eq:FQ-GQ:2} guarantees that we can find 
$\widehat{x}_j'\in\widehat{Q}_j\setminus F_{Q,\varepsilon}$, such that for a.e. $x'\in Q_j$ we have
\begin{align}\label{6t4-EE}
G_Q(x')\leq G_{Q_j}(x')+ H_Q(\widehat{x}_j')\leq G_{Q_j}(x')+\lambda_\varepsilon.
\end{align}
Consequently, $G_{Q_j}(x')>t-\lambda_\varepsilon$ for a.e. $x'\in E_Q(t)\cap Q_j$ which further implies
\begin{equation}\label{7543EDa}
|E_Q(t)\cap Q_j|\leq|\{x'\in Q_j:\,G_{Q_j}(x')>t-\lambda_\varepsilon\}|
\leq\Xi(t-\lambda_\varepsilon)|Q_j|.
\end{equation}
In turn, this permits us to estimate
\begin{multline}
|E_Q(t)|=\sum_j|E_Q(t)\cap Q_j|\leq\Xi(t-\lambda_\varepsilon)\sum_j|Q_j|
\leq\Xi(t-\lambda_\varepsilon)\frac1{\beta}\sum_j |F_{Q,\varepsilon}\cap Q_j|
\\[4pt]
\leq\Xi(t-\lambda_\varepsilon)\frac1\beta|F_{Q,\varepsilon}|
\leq\Xi(t-\lambda_\varepsilon)\frac\alpha\beta|Q|,
\end{multline}
where we have used \eqref{eq:stopping-time}, that the cubes $\{Q_j\}_j$ are pairwise 
disjoint and, finally, \eqref{eq:esti-unif-level-sets}. Dividing by $|Q|$ and taking the 
supremum over all $Q\in\mathbb{D}_{Q_0}$ we arrive at 
\begin{equation}\label{6tgEDac}
\Xi(t)\leq\frac\alpha\beta\Xi(t-\lambda_\varepsilon),\qquad t>\lambda_\varepsilon. 
\end{equation}
Since this is valid for all $\beta \in (\alpha,1)$, we can now let $\beta\to 1^{+}$, iterate the previous expression, and use 
the fact that $\Xi(t)\leq 1$ to conclude that
\begin{equation}\label{RDDSS}
\Xi(t)\leq\frac1{\alpha}\alpha^{\frac{t}{\lambda_\varepsilon}}
=\frac{1}{\alpha}\,e^{-\log(\alpha^{-1})\,\frac{t}{\lambda_\varepsilon}},\qquad t>0. 
\end{equation}
Recalling that $\lambda_\varepsilon=\mathfrak{m}_\alpha+\varepsilon$ and letting $\varepsilon\to 0^{+}$ 
establishes \eqref{eq:JN-point}. 

We shall next indicate how \eqref{eq:JN-point} implies \eqref{eq:exp-L-SI}. Concretely, if we take 
$t:=\frac{1+\alpha^{-1}}{\log(\alpha^{-1})}\,\mathfrak{m}_{\alpha}$ we see that 
\eqref{eq:JN-point} gives
\begin{multline}\label{y54fDXa}
\fint_{Q}\Big(e^{\frac{G_Q(x')}{t}}-1\Big)\,dx'
=\int_0^\infty\frac{|\{x'\in Q:\,G_Q(x')/t>\lambda\}|}{|Q|}e^{\lambda}\,d\lambda
\\[4pt]
\leq\frac{1}{\alpha}\int_0^\infty e^{-\log(\alpha^{-1})\,\frac{\lambda t}{\mathfrak{m}_{\alpha}}} 
e^{\lambda}\,d\lambda=\frac{1}{\alpha}\int_0^\infty e^{-\alpha^{-1}\lambda}\,d\lambda=1.
\end{multline}
With this in hand, \eqref{eq:exp-L-SI} follows with the help of \eqref{eq:Lux-norm}.

At this stage, there remains to justify \eqref{eq:Lq-SI}. This can be done invoking 
again \eqref{eq:JN-point}:
\begin{multline}\label{6rRD}
\fint_{Q} G_Q(x')^q\,dx'=\int_0^\infty\frac{|\{x'\in Q:\,G_Q(x')>\lambda\}|}{|Q|}\,q\,\lambda^q\,\frac{d\lambda}{\lambda}
\leq\frac{1}{\alpha}\int_0^\infty e^{-\log(\alpha^{-1})\,\frac{\lambda}{\mathfrak{m}_{\alpha}}}\,q\,\lambda^q\,
\frac{d\lambda}{\lambda}
\\[4pt]
=\frac{1}{\alpha}\left(\frac{\mathfrak{m}_{\alpha}}{\log(\alpha^{-1})}\right)^q
\int_0^\infty e^{-\lambda}\,q\,\lambda^q\,\frac{d\lambda}{\lambda}
=C_q\frac1{\alpha}\left(\frac{\mathfrak{m}_{\alpha}}{\log(\alpha^{-1})}\right)^q.
\end{multline}
This completes the proof of Proposition~\ref{prop:general-exp-SI}.
\end{proof}

%
  

\def\cprime{$'$} \def\cprime{$'$} \def\cprime{$'$} \def\cprime{$'$}

\end{document}